\documentclass[envcountsame, envcountsect]{svjour3}

\smartqed

\usepackage{epsfig}
\usepackage{pstricks}
\usepackage{url}
\usepackage{amsmath}
\usepackage{amsfonts}

\spnewtheorem*{conjecture*}{Conjecture}{\it}{\rm}

\newcommand{\graph}{\operatorname{G}}
\newcommand{\lk}{\operatorname{lk}}
\newcommand{\st}{\operatorname{st}}

\newcommand{\diam}{{\operatorname{diam}}}
\newcommand{\vdist}{{\operatorname{vdist}}}

\newcommand{\clm}{{\textrm{c.l.m.}}}
\newcommand{\clc}{{\textrm{c.l.c.}}}
\newcommand{\simp}{{\textrm{s}}}
\newcommand{\vol}{\operatorname{Vol}}

\newcommand{\K}{\mathbb K}
\newcommand{\reals}{\mathbb R}
\newcommand{\R}{\mathbb R}
\newcommand{\naturals}{\mathbb N}
\newcommand{\N}{\mathbb N}
\newcommand{\integers}{\mathbb Z}
\newcommand{\Z}{\mathbb Z}

\begin{document} 

\author{Francisco Santos}
\title
{Recent progress on the combinatorial diameter of polytopes and simplicial complexes
\thanks{%
Work of F.~Santos is supported in part by the Spanish Ministry of Science (MICINN) through grant MTM2011-22792  and by the MICINN-ESF EUROCORES programme EuroGIGA - ComPoSe IP04 - Project EUI-EURC-2011-4306.}
}

\institute{
{Departamento de Matem\'aticas, Estad\'istica y Computaci\'on},
{Universidad de Cantabria, E-39005 Santander, Spain}.
\email{francisco.santos@unican.es}, 
}

\titlerunning{Combinatorial diameter of polytopes and simplicial complexes}

\date{\today}

\maketitle

\begin{abstract}

The Hirsch conjecture, posed in 1957, stated that the graph of a $d$-dimensional polytope or polyhedron with $n$ facets cannot have diameter greater than $n-d$. 
%
The conjecture itself has been disproved, but what we know about the underlying question
is quite scarce. Most notably, no polynomial upper bound is known for the diameters that were conjectured to be linear. In contrast, no polyhedron violating the conjecture by more than 25\% is known.

This paper reviews several recent attempts and progress on the question. Some work in the world of polyhedra or (more often) bounded polytopes, but some try to shed light on the question by generalizing it to simplicial complexes. In particular, 
we include here our recent and previously unpublished proof that the maximum diameter of arbitrary simplicial complexes is in 
$n^{\Theta(d)}$ 
and we summarize the main ideas in the {\tt polymath 3} project, a web-based collective effort trying to prove 
an upper bound of type $nd$ for the diameters of polyhedra and of more general objects (including, e.~g., simplicial manifolds).
\end{abstract}



\section{Introduction}

In 1957, Hirsch asked what Dantzig ``expressed geometrically'' as follows~\cite[Problem 13, p.~168]{Dantzig:book}:
\emph{
In a convex region in $n - m$ dimensional space defined by $n$ halfplanes, is $m$ an upper bound for the minimum-length chain of adjacent vertices joining two given vertices?
}
In more modern terminology:

\begin{conjecture*}[Hirsch, 1957]
{The (graph) diameter of a convex polyhedron with at most $n$ facets in $\R^d$ cannot exceed $n-d$.
}
\end{conjecture*}

An unbounded counter-example to this was found by Klee and Walkup in 1967~\cite{KleWal:dstep} and the bounded case was disproved only recently,
by the author of this paper~\cite{Santos:Hirsch-counter}. But the underlying question, how large can the diameter of a polyhedron be in terms of $n$ and $d$, can be considered to be still widely open:

\begin{itemize}
\item 
No polynomial upper bound is known. We can only prove $n^{\log d +2}$ (\emph{quasi-polynomial} bound by Kalai and Kleitman~\cite{KalKle:quasi-polynomial}) and $2^{d-3}n$ (linear bound \emph{in fixed dimension} by Larman~\cite{Larman:upper-bound}, improved to $\frac{2n}{3}2^{d-3}$ by Barnette~\cite{Barnette:upper-bound}). 

\item The known counter-examples violate the Hirsch bound only by a constant and \emph{small} factor (25\% in the case of unbounded polyhedra, 5\% for bounded polytopes). See~\cite{MaSaWe:5prismatoids,Santos:Hirsch-counter}.
\end{itemize}

The existence of a polynomial upper bound is dubbed the ``polynomial Hirsch Conjecture''
and was the subject of the third ``polymath project'', hosted by Gil Kalai in the fall of 2010~\cite{Kalai:polymath3}:

\begin{conjecture}[Polynomial Hirsch Conjecture]\label{conj:hirsch-poly}
Is there a polynomial function $f(n,d)$ such that for any polytope (or polyhedron)
$P$ of dimension $d$ with $n$ facets, $\diam(G(P)) \leq f(n,d)$?
\end{conjecture}

The main motivation for this question is its relation to the worst-case complexity of the simplex method. More precisely, its relation to the possibility that a pivot rule for the simplex method exists that is guaranteed to finish in a polynomial number of pivot steps. In this respect, it is also related to the possibility of designing \emph{strongly polynomial} algorithms for Linear Programming, a problem that Smale included among his list of \emph{Mathematical problems for the 21st century}~\cite{Smale:problems}. See Section~\ref{sec:polyhedra}.

Somehow surprisingly, the last couple of years have seen several exciting results on this problem coming from different directions, 
%
which made De Loera~\cite{DeLoera:optima} call 2010 the \emph{annus mirabilis} for the theory of linear programming.
The goal of this paper is to report on these new results and, at the same time, to fill a gap in the literature since some of them are, as yet, unpublished.

We have decided to make the scope of this paper intentionally limited for lack of time, space, and knowledge, so let us start mentioning several things that we \emph{do not} cover. 
 We do not cover the very exciting new lower bounds, 
 recently found by Friedmann et al.~\cite{Friedmann:Zadeh,FrHaZw:randomized} for the number of pivot steps required by certain classical pivot rules that resisted analysis. 
We also do not cover the promising investigations 
of Deza et al.~on continuous analogues to the Hirsch Conjecture in the context of interior point methods~\cite{DTZ:central-path,DTZ:continuous-dstep,DTZ:curvature} or other attempts at polynomial simplex-like methods for linear 
programming~\cite{Betke:relaxation,Chubanov:binary-solution,DunVem:rescaling,DuSpTe:smoothed,Vershynin:smoothed}.
Information on these developments can be found, apart of the original papers, in~\cite{DeLoera:optima}. (For some of them see also~\cite{Ziegler:whosolved} or our recent survey~\cite{KimSan:update}).

Our object of attention is \emph{the (maximum) diameter of polyhedra} in itself. This may seem a too narrow (and classical) topic, but there have been the following recent results on it, which we review in the second half of the paper:

\begin{itemize}
\item In Section~\ref{sec:smalln} we recall what is known about the exact maximum diameter of polytopes for specific values of $d$ and $n$. In particular, combining clever ideas and heavy computations, Bremner et al~\cite{BDHS:more-bounds,BreSch:diameter-few-facets} have proved that every $d$-polyhedron with at most $d+6$ facets satisfies the Hirsch bound. 

\item In Section~\ref{sec:Hirsch-counter} we give a birds-eye picture of our counter-examples to the bounded Hirsch Conjecture~\cite{MaSaWe:5prismatoids,Santos:Hirsch-counter}, focusing on what can and cannot be derived from our methods.

\item For polytopes and polyhedra whose (dual) face complex is  \emph{flag} the original Hirsch bound holds. (A complex is called \emph{flag} if it is the \emph{clique complex} of a graph). This is a recent result of Adiprasito and Benedetti~\cite{AdiBen:flag-complexes} proved in the general context of flag and \emph{normal} simplicial complexes. See Section~\ref{sec:flag-complexes}.

\item In 1980 Provan and  Billera~\cite{ProBil:decomposable} introduced $k$-decomposa\-bility and weak-$k$-decomposability of simplicial complexes (Definitions~\ref{defi:decomposable} and~\ref{defi:w-decomposable}) in an attempt to prove the Polynomial Hirsch Conjecture: for (weakly or not) $k$-decomposa\-ble polytopes, a bound of type $n^k$ can easily be proved.
Non-$0$-decomposable polytopes were soon found~\cite{KleKle:dstep}, but it has only recently been proved that polytopes that are not \emph{weakly} $0$-decomposable actually exist (De Loera and Klee~\cite{DelKle:decomposable}). In a subsequent paper, H\"aehnle et al.~\cite{HaKlPi:decomposable} extend the result to show that for every $k$ there are polytopes (of dimension roughly $k^2/2$) that are not weakly-$k$-decomposable. See Section~\ref{sec:decomposable}.

\item There is also a recent upper bound for the diameter of a polyhedron in terms of $n$, $d$ and the maximum \emph{subdeterminant} of the matrix defining the polyhedron (where integer coefficients are assumed), proved by Bonifas et al.~\cite{BDEHN:subdeterminants}. Most strikingly, when this bound is applied to the very special case of polyhedra defined by totally unimodular matrices it greatly improves the classical bound by Dyer and Frieze~\cite{DyeFri:random-walks}. See Section~\ref{sec:subdeterminants}.
\end{itemize}

 In the first half we report on some equally recent attempts of settling the Hirsch question by looking at the problem in the more general context of \emph{pure simplicial complexes}:
\emph{
What is the maximum diameter of the dual graph of a simplicial $(d-1)$-sphere or $(d-1)$-ball with $n$ vertices? 
}

Here a simplicial $(d-1)$-ball or sphere is a simplicial complex homeomorphic to the $(d-1)$-ball or sphere. Their relation to the Hirsch question is as follows. It has been known for a long time~\cite{Klee:diameters} that the maximum diameter of polyhedra and polytopes of a given dimension and number of facets is attained at simple ones. Here, a simple $d$-polyhedron is one in which every vertex is contained in exactly $d$-facets. Put differently, one whose facets are ``sufficiently generic''. To understand the combinatorics of a simple polyhedron
$P$  one can look at its dual simplicial complex, which is a $(d-1)$-ball if $P$ is unbounded and a $(d-1)$-sphere if $P$ is bounded and the diameter of $P$ is the \emph{dual diameter} of this simplicial complex.

We can also remove the sphere/ball condition and ask the same for all pure simplicial complexes. In Sections~\ref{sec:complexes} and~\ref{sec:clm} we include two pieces of previously unpublished work in this direction:

\begin{itemize}
\item We construct pure simplicial complexes whose diameter grows exponentially. More precisely, we show that the maximum diameter of simplicial $d$-complexes with $n$ vertices is in $n^{\Theta(d)}$ (Section~\ref{sec:complexes}, Corollary~\ref{coro:high-d}).

\item For complexes in which the dual graph of every star  is connected (the so-called \emph{normal} complexes) the Kalai-Kleitman and the Barnette-Larman bounds stated above hold, essentially with the same proofs. See Theorems~\ref{thm:kk} and~\ref{thm:bl}.
This is a consequence of more general work developed by several people (with a special mention to Nicolai H\"ahnle) in the ``polymath 3'' project coordinated by Gil Kalai~\cite{Kalai:polymath3}. We summarize the main ideas and results of that project in Section~\ref{sec:clm}. The project led to the conjecture that the (dual) diameter of every simplicial \emph{manifold} of dimension $d-1$ with $n$ vertices is bounded above by $(n-1)d$ (Conjecture~\ref{conj:haehnle}).

\end{itemize}

\section{The maximum diameter of simplicial complexes}
\label{sec:complexes}

\subsection{Polyhedra, linear programming and the Hirsch question}
\label{sec:polyhedra}

A \emph{(convex) polyhedron} is a region of $\reals^d$ defined by a finite number of linear inequalities. That is,
the set
\[
P(A, {\bf b}):= \{ {\bf x} \in \R^d : A{\bf x} \le {\bf b}\},
\]
for a certain real $n\times d$ matrix $A$ and right-hand side vector ${\bf b}\in \reals^n$. A bounded polyhedron is a \emph{polytope}.
Polyhedra and polytopes are the geometric objects underlying Linear Programming. Indeed, the \emph{feasibility region} of a linear program 
\[
\text{ Maximize } c\cdot {\bf x}, \text{ subject to } A{\bf x} \le {\bf b}
\]
is the polyhedron $P(A, {\bf b})$.

The combinatorics of a polytope or polyhedron is captured by its lattice of \emph{faces}, where a face of $P(A, {\bf b})$ is the set where a linear functional is maximized. 
%
Faces of a polyhedron $P$ are themselves polyhedra of dimensions ranging from $0$ to $\dim(P) -1$. 
Those of dimensions $0$, and $\dim(P) -1$ are called, respectively, \emph{vertices} and \emph{facets}. Bounded faces of dimension $1$ are called \emph{edges} and unbounded ones \emph{rays}. The vertices and edges of a polyhedron $P$ form the \emph{graph} of $P$, which we denote $\graph(P)$. We are interested in the following question:

\begin{question}[Hirsch question]
What is the maximum diameter among all polyhedra with a given number $n$ of facets and a given dimension $d$?
\end{question}

We call $H_p(n,d)$ this maximum. The Hirsch Conjecture stated that
\[
H_p(n,d)\le n-d.
\]

Hirsch's motivation for raising this question (and everybody else's for studying it!) is that 
the celebrated \emph{simplex algorithm} of 
George Dantzig~\cite{Dantzig:simplex}, one of the ``ten algorithms with the greatest influence in the development of science and engineering in the 20th century'' according to the list compiled by Jack Dongarra and Francis Sullivan~\cite{DonSul:topten,Nash:simplex-method}, solves linear programs by walking along the graph of the feasibility region, from an initial vertex that is easy to find (perhaps after a certain transformation of the program which does not affect its optimum) up to an optimal vertex.
When an improving edge does not exist we have either found a ray where the functional is unbounded (proving that the LP has no optimum) or a local and, by convexity, global optimum. 

That is to say, the Hirsch question is closely related to the worst-case computational complexity of the simplex method, a question that is somehow open; we know the simplex method to be exponential or subexponential (in the wort case) with most of the pivot rules that have been proposed, where a \emph{pivot rule} is the rule used by the method to choose the improving edge to follow. (See~\cite{KleMin:exponential-simplex} for Dantzig's \emph{maximum gradient} rule, the first one that was solved, and \cite{Friedmann:Zadeh,FrHaZw:randomized} for some recent additions, including \emph{random edge}, \emph{random facet} and Zadeh's \emph{least visited facet} rule).
But we do not know whether polynomial pivot rules exist. Of course, this is impossible (or, at least, it would require some delicate strategy to find a good initial vertex) if the answer to the Hirsch question turns out to be that $H_p(n,d)$ is not bounded by a polynomial. 

Observe also that, even if polynomial-time algorithms for linear programming are known (the most classical ones being the ellipsoid method of Khachiyan (1979) and the interior point method of Karmarkar (1984)), all of them work by successive approximation and, hence, they are not \emph{strongly polynomial}. They are polynomial in bit-complexity when the bit-size of the input is taken into account, but they are not guaranteed to finish in a number of arithmetic operations that is polynomial in $n$ and $d$ alone. In contrast, a polynomial pivot rule for the simplex method would automatically yield a strongly polynomial algorithm for linear programming. This relates the
Polynomial Hirsch Conjecture to one of Smale's ``Mathematical problems for the 21st century''~\cite{Smale:problems}, namely the existence of strongly polynomial algorithms for linear programming.

\subsection{From polyhedra to simplicial complexes}

It is known since long that the maximum $H_p(n,d)$ is achieved at a \emph{simple} polyhedron, for every $n$ and $d$~\cite{Klee:diameters}. 
So, let $P$ be a simple $d$-polyhedron with $n$ facets, which we label (for example) with the numbers $1$ to $n$. Since each non-empty face of $P$ is (in a unique way) an intersection of facets, we can label faces as subsets of $[n]:=\{1,\dots,n\}$. Let $C$ be the collection of subsets of $[n]$ so obtained. It is well-known (and easy to prove) that, if $P$ is simple, then 
$C$ is a  \emph{pure (abstract) simplicial complex} of dimension $d-1$ with $n$ vertices (or, a \emph{$(d-1)$-complex} on $n$ vertices, for short). 

Here, a simplicial complex is a collection $C$ of subsets from a set $V$ of size $n$ (for example, the set $[n]=\{1,\dots,n\}$) with the property that if $X$ is in $C$ then every subset of $X$ is in $C$ as well. The individual sets in a simplicial complex are called the \emph{faces} of it, and the maximal ones are called its \emph{facets}. A  complex is \emph{pure of dimension $d-1$} if all its facets  have the same cardinality, equal to $d$. Of course, a pure simplicial complex is determined by its list of facets alone, which are subsets of $[n]$ of cardinality $d$. Conversely, any such subset defines a pure simplicial complex of dimension $d-1$. Hence we take this as a definition:

\begin{definition}
A \emph{pure simplicial complex of dimension $d-1$}  (or, a \emph{$(d-1)$-complex})
is any family of $d$-element subsets of an $n$-element set $V$ (typically, $V=[n]:=\{1,\dots,n\}$).  Elements of $C$ are called \emph{facets} and any subset of a facet is a \emph{face}. More precisely, a $k$-face is a face with $k+1$ elements. Faces of dimensions $0$, $1$, and $d-2$ are called, respectively, \emph{vertices}, \emph{edges} and \emph{ridges}.
\end{definition}

Observe that a pure $(d-1)$-complex is the same as a \emph{uniform hypergraph of rank $d$}. Its facets are called \emph{hyperedges} in the hypergraph literature.

The particular $(d-1)$-complex obtained above from a simple polytope $P$ is called the \emph{dual face complex} of $P$.
The \emph{adjacency graph} or \emph{dual graph} of a pure complex $C$, denoted $\graph(C)$, is the graph having as vertices the facets of $C$ and as edges the pairs of facets $X,Y\in C$ that differ in a single element (that is, those that share a ridge). 
We are only interested in complexes with a connected adjacency graph, which are called \emph{strongly connected} complexes.

\begin{example}
\rm
A $0$-complex is just a set of elements of $[n]$ and its adjacency graph is complete. A $1$-complex is a graph $G$, and its adjacency graph is the \emph{line graph} of $G$. 
The adjacency graph of the \emph{complete complex} $\binom{[n]}{ d}$ is usually called the \emph{Johnson graph} $J(n,d)$ (see, e.g., \cite{HolShe:Petersen}). It is also the graph (1-skeleton) of the $d$-th hypersimplex of dimension $n-1$~\cite{DeRaSa:triang_book,Ziegler:lectures-on-polytopes}, and the basis exchange graph of the uniform matroid of rank $d$ on $n$ elements.
\end{example}

We leave it to the reader to check the following elementary fact:

\begin{proposition}
\label{prop:dual-graph}
Let $P$ be a simple polyhedron with at least one vertex. Then the adjacency graph of the dual complex of $P$ equals the graph of $P$.
\qed
\end{proposition}

As usual, the \emph{diameter} of a graph is the maximum distance between its vertices, where the distance between vertices is the minimum number of edges  in a path from one to the other. For simplicity, we abbreviate ``diameter of the adjacency graph of the complex $C$'' to ``diameter of $C$''. The main object of our attention in this section is the function
\[
H_\simp(n,d):=\text{maximum diameter of strongly connected $(d-1)$-complexes on $[n]$}.
\]


\begin{proposition}
\label{prop:Johnson}
$H_\simp(n,d)$ equals the length of the longest induced path in the Johnson graph $J(n,d)$.
\end{proposition}

\begin{proof}
Every pure $(d-1)$-simplicial complex $C$ on $n$ vertices is a subcomplex of the complete complex $\binom{[n]}{ d}$, and the adjacency graph of $C$ is the corresponding induced subgraph in $J(n,d)$. So, it suffices to show that $H_\simp(n,d)$ is achieved at a complex whose adjacency graph is a path.

For this, let $C$ be any pure $d$-complex, and let $X$ and $Y$ be facets at maximal distance. Let $\Gamma$ be a shortest path of facets from $X$ to $Y$. Then $\Gamma$, considered as a set of facets, is a pure $d$-complex with the same diameter as $C$ and its adjacency graph is a path (or otherwise the path $\Gamma$ in $C$ would not be shortest).
\qed
\end{proof}

\begin{remark}
\rm
There is some literature on the problem of finding the longest induced path in  an arbitrary graph, which is NP-complete in general. 
But we do not know of any where this problem is addressed for the Johnson graph specifically.
\end{remark}

We call complexes whose adjacency graphs are paths \emph{corridors}. They are particular examples of \emph{pseudomanifolds}, that is complexes in which every ridge is contained in at most two facets. 

\begin{corollary}
\label{coro:corridor1}
The maximum diameter of pure simplicial complexes of fixed dimension and number of vertices is always attained at a corridor, hence at a pseudomanifold. 
In particular,
\[
H_\simp(n,d) < \frac{1}{d-1}\binom{n}{ d-1} -1.
\]
\end{corollary}

\begin{proof}
Let $P$ be a $(d-1)$-corridor on $n$ vertices and of length $N$. We can explicitly compute its number of ridges as follows: each of the $N+1$ facets has $d$ ridges, and exactly $N$ of these ridges belong to two facets, the rest only to one. Thus, the number of ridges equals
\[
d(N+1)-N = N(d-1) +d.
\]
Now, the total number of ridges in $P$ is at most $\binom{n}{ d-1}$ (the ridges in the complete complex), so
\[
N(d-1) +d \le \binom{n}{d-1},
\]
or
\[
N\le \frac{1}{d-1}\binom{n}{ d-1} -\frac{d}{d-1} < \frac{1}{d-1}\binom{n}{ d-1} -1.
\]
\qed
\end{proof}

\subsection{The maximum diameter of simplicial complexes}


\begin{theorem}
\label{thm:dimension2}
\[
\frac{2}{9}(n-3)^2 \le H_\simp(n,3) \le \frac{1}{4}n^2.
\]
\end{theorem}

\begin{proof}
The upper bound follows immediately from Corollary~\ref{coro:corridor1}.
%
For the lower bound let us first assume that $n=3k+1$ (that is, $n$ equals $1$ modulo $3$), and show:
\[
H_\simp(3k+1,d) \ge 2k^2+k-2 > 2k^2=\frac{2}{9}(n-1)^2.
\]
For this, separate the set of $3k+1$ vertices into two parts $V_1$ and $V_2$ with $|V_1|=2k+1$ and $|V_2|=k$.
With the first $2k+1$ vertices we construct $k$ edge-disjoint cycles; that is, we decompose the complete graph $K_{2k+1}$ into $k$ Hamiltonian cycles. This can be done by rotating in all $k$ possible ways the cycle shown (for $k=4$) in Figure~\ref{fig:cycle}.
\begin{figure}
\includegraphics[scale=0.5]{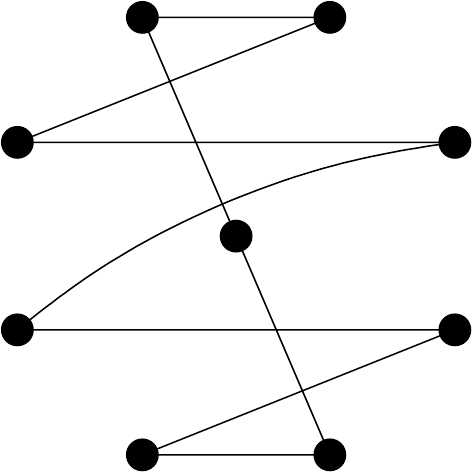}
\caption{Decomposing $K_{2k+1}$ into $k$ Hamiltonian cycles}
\label{fig:cycle}
\end{figure}
We now remove one edge from each of these cycles, so as to construct a walk of $2k^2$ edges with the property that each of the $k$ sections of length $2k$ in it does not repeat vertices. This can be done by arbitrarily removing an edge $i_0i_1$ in the first cycle, then removing one of the two edges incident to $i_1$ in the second cycle (call $i_2$ the other end of that edge), then one incident to $i_2$ in the third cycle, etc.

Then a corridor is constructed as shown in Figure~\ref{fig:corridor}: each of the paths of length $2k$ is joined to one of the $k$ vertices of $V_2=\{j_1,j_2,\dots,j_k\}$ that we have not yet used, and the different sections are glued together with $k$ additional triangles. That this is a corridor, of length $2k^2+k-2\ge \frac{2}{9}(n-1)^2$, follows from the fact that no edge is used twice, so that the adjacency graph is indeed the path that we see in the figure. 
\begin{figure}
\resizebox{0.6\linewidth}{!}{
  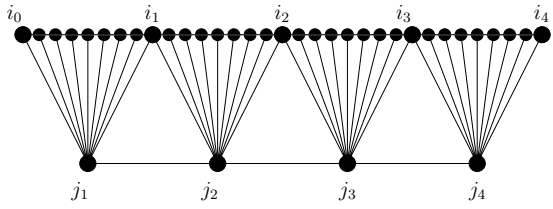
}
\caption{A $2$-corridor of length $\frac{2}{9}(n-1)^2$}
\label{fig:corridor}
\end{figure}
If $n=3k+2$ or $3k+3$ we neglect one or two points of $n$.
\qed
\end{proof}

In order to get lower bounds in higher dimension we use the \emph{join} operation. Remember that the join of two simplicial complexes $C_1$ and $C_2$ on disjoint sets of vertices is
\[
C_1*C_2:=\{v_1\cup v_2 : v_1 \in C_1, v_2\in C_2\}.
\] 
The join of complexes of dimensions $d_1-1$ and $d_2-1$ has dimension $d_1+d_2-1$ and the adjacency graph of a join is the cartesian product of the adjacency graphs of the factors; that is
\[
\graph(C_1*C_2)=\graph(C_1) \square \graph(C_2),
\]
where the Cartesian product $\graph(C_1) \square \graph(C_2)$ of two graphs $G_1=(V_1,E_1)$ and $G_2=(V_2,E_2)$ is the graph with vertex set $V_1\dot\cup V_2$ in which $(u_1,u_2)$ is adjacent to $(v_1,v_2)$ if either $u_1=v_1$ and $(u_2,v_2)\in E_2$ or $u_2=v_2$ and $(u_1,v_1)\in E_1$.

\begin{lemma}
\label{lemma:join}
Let $G_1$ and $G_2$ be paths of lengths $l_1$ and $l_2$. Then $G_1\square G_2$ has an induced path of length at least $l_1l_2/2$.
\end{lemma}

\begin{proof}
Let the vertices of $l_1$ be $v_1,v_1,v_2,\dots,v_{l_1}$, numbered in the order they appear along the path. Consider every second vertical path $v_i\times G_2$, for even $i=0,2,\dots,\lfloor{l_1/2}\rfloor$ and join these paths to one another by horizontal paths of length two, alternating between the beginning and end of $l_2$ (see the left part of Figure~\ref{fig:snake}). 
This is an induced path in $G_1\square G_2$ of length
\[
(\lfloor l_1/2\rfloor+1)l_2+l_1\ge\frac{l_1l_2}{2}.
\]
\qed
\end{proof}
\begin{figure}
\includegraphics[scale=0.6]{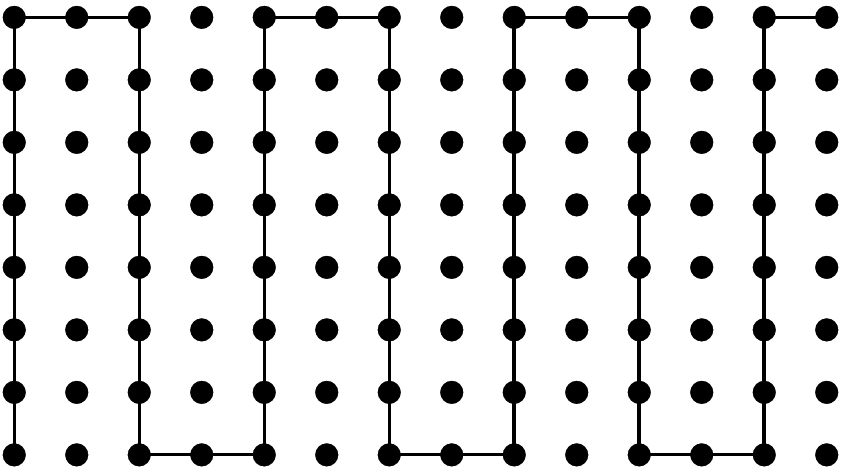}
\qquad \qquad
\includegraphics[scale=0.6]{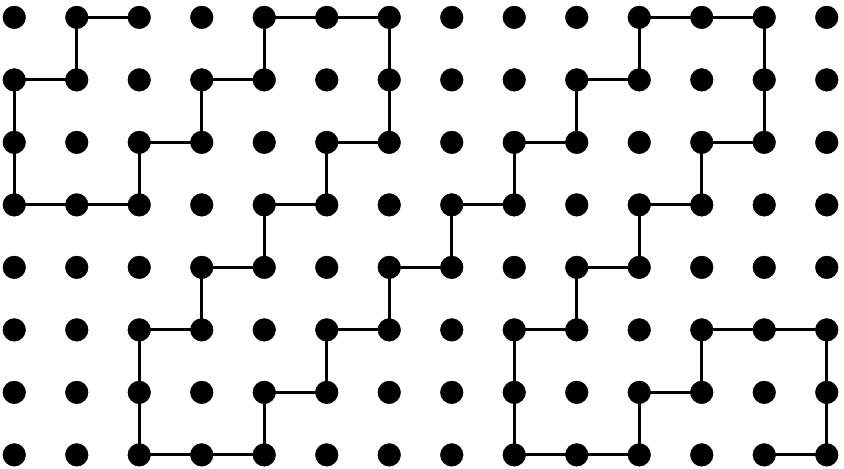}
\caption{Two induced paths in the Cartesian product of two paths}
\label{fig:snake}
\end{figure}

\begin{remark}[Personal communication of G.~Rote]
\rm
The path of the previous lemma uses about  $\frac{1}{2}$ of the vertices in the Cartesian product. This can be improved to $\frac{2}{3}$ using diagonal zig-zag paths instead of vertical paths (see the right part of Figure~\ref{fig:snake}). Indeed, zig-zag paths cover almost completely every two of each 3 diagonal lines of vertices. 
Moreover, $\frac{2}{3}$ is optimal since an induced path in a graph of maximum degree $k$ cannot use more than $k/(2k-2)$ of the vertices. 
\end{remark}

\begin{corollary}
\label{coro:join}
For all $n,d,k\in \naturals$:
\[
H_\simp(kn,kd)\ge \frac{H_\simp(n,d)^k}{2^{k-1}}.
\]
\end{corollary}

\begin{proof}
By induction on $k$, applying Lemma~\ref{lemma:join}  to corridors that achieve the maximum lengths $H_\simp(n,d)$ and $H_\simp((k-1)n,(k-1)d)$. 
\qed
\end{proof}

\begin{corollary}
\label{coro:high-d}
\[
2\left(\frac{1}{3}\left\lfloor\frac{3n}{d}\right\rfloor-1\right)^{2d} \le H_\simp(3n,3d) \le \frac{1}{3d-1}\binom{3n}{ 3d-1}.
\]
In particular, in fixed dimension $d-1$:
\[
\Omega\left(n^{\frac{2d}{3}}\right) \le H_\simp(n,d) \le O(n^{d-1}).
\]
\end{corollary}

\begin{proof}
The upper bound is Corollary~\ref{coro:corridor1}. For the lower bound, Corollary~\ref{coro:join} and Theorem~\ref{thm:dimension2} give:
\[
H_\simp(3n,3d)\ge \frac{H_\simp(\lfloor\frac{3n}{d}\rfloor,3)^d}{2^{d-1}}\ge  \frac{(\frac{2}{9}\left(\left\lfloor\frac{3n}{d}\right\rfloor-3\right)^2)^d}{2^{d-1}}.
\]
\qed
\end{proof}

%






%

\section{Connected layer (multi)-complexes}
\label{sec:clm}

\subsection{A tempting conjecture}

Corollary~\ref{coro:high-d} implies that if we want to prove polynomial diameters for the graphs of simple polytopes (that is, for the adjacency graphs of simplicial polytopes) we cannot hope to do it in the general framework of simplicial complexes. Some combinatorial, topological, or geometric restriction needs to be put on the complexes under scrutiny. Corollary~\ref{coro:corridor1} says that being a pseudomanifold is not enough (in fact, it is no ``loss of generality'') but perhaps being a manifold is. One property that manifolds have and which seems promising is:

\begin{definition}
A pure simplicial complex is called \emph{normal}~\cite{AdiBen:flag-complexes} or \emph{locally strongly connected}~\cite{JosIzm:branched-coverings}
if the link (equivalently, if the star) of every face is strongly connected.
\end{definition}

Here the \emph{link} and the \emph{star} of a face $S$ in a complex $C$ are defined as
\[
\lk_C(S):=\{X\setminus S: S\subseteq X \in C\},
\qquad
\st_C(S):=\{X: s\subseteq X \in C\}.
\]
Clearly, $\st_C(s)= S * \lk_C(S)$. If $C$ is pure then the star and link of every face $S$ are pure, of dimensions $\dim(C)$, and $\dim(C)-|S|$, respectively.

The following remarks are evidence that normality of a simplicial complex is a natural property for studying the Hirsch question:
\begin{itemize}
\item All simplicial spheres and balls and, in fact, all simplicial manifolds, with or without boundary, are normal complexes.

\item Adiprasito and Benedetti~\cite{AdiBen:flag-complexes} (see also Section~\ref{sec:flag-complexes}) have recently shown that normal and \emph{flag} pure simplicial complexes satisfy the Hirsch bound,
where a simplicial complex is called flag if it equals the clique complex of its 1-skeleton. 
Unfortunately, face complexes of simplicial polytopes may not be flag, so this condition is too restrictive for proving the Polynomial Hirsch Conjecture. 

\item The best upper bounds for the diameters of polytopes that we know of (the Kalai-Kleitman bound of $n^{O(\log d)}$ and the Barnette-Larman bound of $O(n2^d)$) can be easily proved for arbitrary normal complexes, as we show below.
\end{itemize}

In what follows, we report on the ideas and partial results obtained in the \texttt{polymath 3} project~\cite{Kalai:polymath3} started by Gil Kalai in October 2010. The main goal of the project was to prove a polynomial bound for the diameter of normal simplicial complexes. But the setting was the following generalization of normal complexes:

\begin{definition}
\rm
A \emph{pure multicomplex} of rank $d$ on $n$ elements is a collection $M$ of multisets of size $d$ of $[n]$ (or of any other set $V$ of size $n$). Here, a multiset is an unordered set of elements of $[n]$ with repetitions allowed. Formally, a multiset of size $d$ can be modeled as a degree $d$ monomial in $\K[x_1,\dots,x_n]$. 
\end{definition}

We keep for multicomplexes the notions defined for complexes, such as \emph{facet}, \emph{face}, \emph{link}, etc. For example, the star and the link of a face $S$ in $M$ are
\[
\lk_M(S):=\{X\setminus S: S\subseteq X \in C\},
\qquad
\st_M(S):=\{X: s\subseteq X \in C\}.
\]
But, of course, set operations have to be understood in the multiset sense. The intersection of two multisets $A$ and $B$ is the multiset that contains each element of $[n]$ with the minimum number of repetitions it has in $A$ or $B$, and the union is the same with minimum replaced to maximum. If multisets are modeled as monomials, intersection and union become $\gcd$ and $\operatorname{lcm}$. We also keep for multicomplexes the definitions of dual graph, diameter, and of being strongly connected or normal.

We now get to the main definition in this section:

\begin{definition}
\label{defi:clm}
A \emph{connected layer multicomplex} (or \emph{\clm}, for short), is a pure multicomplex $M$ together with a partition $M=L_a\dot\cup L_{a+1}\dot\cup 
\dots \dot\cup L_{b}$ (with $a,b\in \integers$, $a\le b$) of its set of facets into \emph{layers} having the following connectedness property:
\begin{equation}
\text{For every mutisubset $S$, the star of $S$ intersects an interval of layers.}
\label{eq:clm}
\end{equation}
The \emph{length} of a \clm~is $b-a$, that is, one less than the number of layers.
\end{definition}

In particular, we are interested in the function:
\[
H_\clm(n,d):=\text{maximum length of \clm's of rank $d$ with $[n]$ elements}.
\]

This definition, and some of the results below, were first introduced by Eisenbrand, H\"ahnle, Razborov, and Rothvo\ss~\cite{EHRR:limits-of-abstraction} except they considered usual \emph{complexes} instead of \emph{multicomplexes}. The generalization to multicomplexes was proposed by H\"ahnle in the 
 \texttt{polymath 3} project~\cite{Kalai:polymath3}, where he also made the following conjecture:

\begin{conjecture}[H\"ahnle-polymath 3~\cite{Kalai:polymath3}]
\label{conj:haehnle}
\[
\forall n,d, \qquad H_\clm(n,d)\le d(n-1).
\]
\end{conjecture}


\begin{remark}
This conjecture would imply, by Proposition~\ref{prop:clm} below, the same bound for the diameter of every normal simplicial multicomplex of dimension $d-1$ on $n$ vertices. In particular, for the graph-diameter of $d$-polytopes with $n$ facets. 
\end{remark}

Let us be more explicit about condition~(\ref{eq:clm}). What we mean is that for every multiset $S$, if there are facets that contain $S$ in layers $L_i$ and $L_j$, then there is also a facet containing $S$ in every intermediate layer. This easily implies:

\begin{proposition}
Let $M$ be a \clm~of rank $d$ on $n$ elements. Then for every face $S$, the link of $S$ in $M$ is a \clm~of rank $d-|S|$ on (at most) $n$ elements, where the $i$-th layer of $\lk_M(S)$ is defined to be $\lk_{L_i}(S)$.
\qed
\end{proposition}

Every normal multicomplex $M$ can naturally be turned into a \clm~of length equal to the diameter of $M$. For this, let $X$ and $Y$ be facets of $M$ at distance equal to the diameter of $M$. We layer $M$ by ``distance to $X$''. That is, we let $L_i$ contain all the facets that are at distance $i$ to $X$ in the adjacency graph of $M$ (e.g., $L_0=\{X\}$).  Normality of $M$ implies the connectedness condition (\ref{eq:clm}).
Hence:

\begin{proposition}
\label{prop:clm}
The maximum diameter among all normal multicomplexes of rank $d$ with $n$ elements is smaller or equal than $H_\clm(n,d)$.
\qed
\end{proposition}

\subsection{Two extremal cases}

We call a \clm~\emph{complete} if its underlying multicomplex is complete; that is, each of the $\binom{n+d-1 }{ d}$ multisubsets of $[n]$ of size $d$ is used in some layer. We call it \emph{injective} if each layer has a single facet; that is, the map $M\to \integers$ that assigns facets to its layers is injective.
These two classes of \clm's are extremal and opposite, in the sense that the complete \clm's have the maximum number of facets for given $n$ and $d$ and injective ones the minimum possible number for a given length. 

\begin{example}[A complete \clm, polymath 3~\cite{Kalai:polymath3}]
\label{exm:complete-clm}
For any $d$ and $n$, consider the 
complete multicomplex of rank $d$ on the set $[n]$.
Consider it layered putting in layer $i$ ($i=d,\dots,nd$) all the facets with sum of elements equal to $i$. This is a \clm{} of length $d(n-1)$.
\end{example}

\begin{example}[An injective \clm, polymath 3~\cite{Kalai:polymath3}]
\label{exm:injective-clm}
For any $d$ and $n$, consider the multicomplex of rank $d$ on $[n]$ consisting of all the multisets using at most two different elements from $[n]$, and consecutive ones. For example, for $n=4$, $d=3$ our multicomplex is
\[
\{111,112,122,222,233,233,333,334,344,444\}.
\]
Consider it layered with the restriction of the layering in the previous example.
This produces an injective \clm~of the same length $d(n-1)$.
\end{example}

\begin{proposition}[polymath 3~\cite{Kalai:polymath3}]
\label{prop:extremal-clm}
\label{prop:injective-clm}
\label{prop:complete-clm}
Let $M=L_a\dot\cup\dots\dot\cup L_b$ be a connected layer multicomplex of rank $d$ on $n$ elements. If $M$ is either complete or injective, then its length is at most $d(n-1)$.
\end{proposition}

\begin{proof}
%
If $M$ is complete, we proceed by induction on $d$, the case $d=1$ being trivial.
Let $X\in L_a$ and $Y\in L_b$ be facets in the first and last layer, respectively, and let $i$ and $j$ be elements in $X\setminus Y$ and $Y\setminus X$ respectively (these formulas have to be understood in the multiset sense. That is, $i\in X\setminus Y$ means that $i$ appears more times in $X$ than in $Y$). Let $X'=X\setminus \{i\} \cup \{j\}$, which must be in some layer $L_c$. Since $X$ and $X'$ differ on a single element, $c-a\le n-1$ (the link of $X\cap X'$ in $M$ is a \clm of rank $d$). On the other hand, the link of $j$ in $M$ is a \clm~of rank $d-1$ on $n$ elements and it intersects (at least) the layers from $c$ to $b$ of $M$, so that $b-c\le (n-1)(d-1)$. Putting this together:
\[
b-a = (b-c) + (c-a) \le (n-1)(d-1) + (n-1) = d(n-1).
\]

For the injective case we observe that the degree of each element $i\in[n]$ in the sequence of layers of $M$ is a unimodal function: it (weakly) increases up to a certain point and then it decreases. In particular, there are at most $d$ steps where the degree of $i$ increases from one layer to the next. On the other hand, at least one degree increases at each step, so the number of steps is at most $dn$. This bound decreases to $dn-d$ if we observe that in the first layer some degrees where already positive; more precisely, the initial sum of degrees is exactly $d$.
\qed
\end{proof}

Proposition~\ref{prop:extremal-clm} is quite remarkable. It shows that in two ``extremal and opposite'' cases of connected layer multicomplexes  we have an upper bound of $d(n-1)$ for their length. And Examples~\ref{exm:complete-clm} and~{exm:injective-clm} show that this bound is attained in both cases. This is, in our opinion,  what makes Conjecture~\ref{conj:haehnle} exciting.

\subsection{Two upper bounds}

Here we show that the two best upper bounds on diameters of polytopes that we know of can actually be proved in the context of \clm's.

\begin{lemma}
\label{lemma:kk}
For every $n$ and $d$ we have
\[
H_\clm(n,d) \le H_\clm\left(\left\lfloor \frac{n-1}{2}\right\rfloor, d \right)+ H_\clm \left(\left\lceil \frac{n-1}{2} \right\rceil, d \right) + H_\clm(n,d-1) + 2.
\]
\end{lemma}

\begin{proof}
Let $M$ be a \clm~of rank $d$ on $n$ elements. 
Let $i$ be the largest integer such that the first $i$ layers of $M$ use at most $\lfloor (n-1)/2\rfloor$ of the $n$ elements. 
Let $j$ be the largest integer such that the last $j$ layers of $M$ use at most $\lceil (n-1)/2\rceil$ of the $n$ elements. 
Let $k$ be the remaining number of layers. By construction, there has to be some common element used in all these $k$ layers. Hence:
\[
i-1 \le H_\clm(\lfloor n/2\rfloor, d),\quad  j-1 \le H_\clm(\lceil n/2\rceil, d), \quad  k-1 \le H_\clm(n, d-1).
\]
This gives the bound, since the length of our \clm{} equals
$ i+j+k-1 $.
\qed
\end{proof}

\begin{theorem}[Kalai-Kleitman~\cite{KalKle:quasi-polynomial}, Eisenbrand et al.~\cite{EHRR:limits-of-abstraction}]
\label{thm:kk}
\[
H_\clm(n,d) \le n^{\log_2 d +1}-1
\]
\end{theorem}

\begin{proof}
We assume that both $n$ and $d$ are at least equal to $2$, or else the statement is trivial (and the inequality is tight). For later use we observe that, under these constraints:
\begin{equation}
\label{eg:kk1}
{n}\le 2\frac{n^2}{4} = 2\cdot 4^{\log_2 n -1} \le  2(2d)^{\log_2 n -1}.
\end{equation}

Let us now convert the inequality of Lemma~\ref{lemma:kk} in something more usable. Letting $f(n,d):=H_\clm(n,d) +1$ and taking into account that $\left\lfloor \frac{n-1}{2}\right\rfloor\le \left\lceil \frac{n-1}{2} \right\rceil \le \left\lfloor \frac{n}{2}\right\rfloor$ (plus monotonicity of $H_\clm$) we get
\[
f(n,d)\le 2f(\left\lfloor{n/2}\right\rfloor,d) + f(n,d-1).
\]

Applying this recursively we get
\begin{equation}
\label{eg:kk2}
f(n,d)\le 2\sum_{i=2}^df(\left\lfloor{n/2}\right\rfloor,i) + f(n,1)\le 2(d-1) f(\left\lfloor{n/2}\right\rfloor,d) + n.
\end{equation}

Now, by inductive hypothesis,
\[
f(\left\lfloor{n/2}\right\rfloor,d)\le\left(\frac{n}{2}\right)^{\log_2 d +1} =(2d)^{\log_2 n -1}.
\]

Plugging this into~(\ref{eg:kk2}), and using~(\ref{eg:kk1})  gives
\[
f(n,d)\le 2(d-1) (2d)^{\log_2 n -1} + n \le (2d)^{\log_2 n}= n^{\log_2 2d}.
\]
%
\qed
\end{proof}

\begin{lemma}
\label{lemma:bl}
For every \clm~$M$ of rank $d$ on $n$ elements there are $n_1,\dots, n_k$ with $\sum n_i \le 2n-1$ and such that the length of $M$ is bounded above by
\[
H_\clm(n_1,d-1) + \dots + H_\clm(n_k,d-1) + k-1.
\]
\end{lemma}

\begin{proof}
Let $L_0,\dots, L_N$ be the layers of $M$. Let $l_1$ be the last layer such that $L_0$ and $L_{l_1}$ use some common vertex, and let $n_1$ be the number of vertices used in 
the layers $L_0\cup\dots\cup L_{l_1}$. Clearly, $l_1\le H_\clm(n_1,d-1)$. Apply the same to the remaining layers $L_{l_1+1},\dots,L_N$. That is to say, let $l_2$ be the last layer that 
shares a vertex with $L_{l_1+1}$ and let $n_2$ be the vertices used in $L_{l_1+1},\dots,L_{l_2}$,
etc. 

At the end we have decomposed $M$ into several (say $k$) connected layer multicomplexes so, indeed, its 
length is the sum of the $k$ lengths plus $k-1$. Since the $i$-th sub-multicomplex has the property that it uses $n_i$ elements and one of them is used in all layers, its length is at most $H_\clm(n_i,d-1)$. 

It only remains to be shown that $\sum n_i \le 2n-1$. That $\sum n_i \le 2n$ comes from the fact that
no vertex can be used in more that two of the sub-multicomplexes, by construction of them. The $-1$ from the fact that vertices
used in $L_0$ cannot be used in any sub-multicomplex other than the first one.
\qed
\end{proof}

\begin{theorem}[Larman~\cite{Larman:upper-bound}, Barnette~\cite{Barnette:upper-bound}, Eisenbrand et al.~\cite{EHRR:limits-of-abstraction}]
\label{thm:bl}
\[
H_\clm(n,d) \le  (n-1)2^{d-1}. 
\]
\end{theorem}

\begin{proof}
By induction on $d$, with the case $d=1$ being trivial. For the general case, let $M$ be a \clm~of maximal length equal to $H_\clm(n,d)$ and use the decomposition of Lemma~\ref{lemma:bl}. This gives:
\begin{eqnarray*}
H_\clm(n,d) &\le &\sum_{i=1}^k H_\clm(n_i,d-1) + k-1\\
&\le& \sum_{i=1}^k (n_i-1)2^{d-2} + k-1\\
&\le& (2n-k-1)2^{d-2} + k-1 \\
&= &(n-1)2^{d-1}-(k-1)(2^{d-2}-1) \le (n-1)2^{d-1}.
\end{eqnarray*}
\qed
\end{proof}

{\color{red}
}

\begin{corollary}
\label{coro:n=3}
\label{coro:smalln-d}
If $n\le 3$ or $d\le 2$, then
\[
H_\clm(n,d)=(n-1)d.
\]
\end{corollary}

\begin{proof}
By Examples~\ref{exm:complete-clm} and~\ref{exm:injective-clm} we only need to prove the upper bound. In the cases 
$n\le 2$ or $d=1$ we have that $(n-1)d+1=\binom{n+d-1}{d}$, which is the number of facets in the complete multicomplex, so the bound is trivial. The case $d=2$ follows from Theorem~\ref{thm:bl} and the case $n=3$ from Lemma~\ref{lemma:kk}.
%
\qed
\end{proof}

Besides the values in this corollary, N.~H\"ahnle~\cite{Kalai:polymath3} has  verified Conjecture~\ref{conj:haehnle} for all values of $(n,d)$ in or below $\{(4,13), (5,7), (6,5), (7,4), (8,3) \}$.

\subsection{Variations on the theme of \clm's}
What is contained above are the main properties and results on connected layer multicomplexes and, in particular, the main outcome of the polymath 3 project. But it may be worth mentioning other related ideas, questions and loose ends.

\subsubsection*{Complexes versus multicomplexes}

Connected layer multicomplex were introduced in~\cite{Kalai:polymath3} based on previous work of Eisenbrand, H\"ahnle, Razborov, and Rothvo\ss~\cite{EHRR:limits-of-abstraction} in which they introduced \emph{connected layer complexes} (under the name connected layer \emph{families}). The definition is exactly the same as Definition~\ref{defi:clm} except $M$ is now a pure simplicial complex, rather than a multicomplex.
Since the concept is more restricted, all the upper bounds that we proved for \clm's 
(Proposition~\ref{prop:extremal-clm}, Theorems~\ref{thm:kk} and~\ref{thm:bl}) are still valid
for connected layer complexes. The question is whether multicomplexes allow for longer objects than complexes. The following result says that ``not much more''. In it, $H_\clc(n,d)$ denotes the maximum length of connected layer complexes of rank $d$ on $n$ elements.

\begin{theorem}[polymath 3~\cite{Kalai:polymath3}]
\label{thm:multi}
\[
H_\clc(n,d) \le H_\clm(n,d) \le H_\clc(nd,d).
\]
\end{theorem}

\begin{proof}
The first inequality is obvious. For the second one, let $M$ be a \clm{} of rank $d$ on $n$ elements achieving $H_\clm(n,d)$. We can construct a connected layer complex of the same length and rank
on the set $[n]\times [d]$ simply by the following substitution of multisets to sets:
\[
\{1^{k_1},\dots, n^{k_n}\} \mapsto \{(1,1),\dots,(1,{k_1}),\dots,  (n,1),\dots, (n,{k_n})\}.
\]
\qed
\end{proof}

As further evidence, one of the main results of~\cite{EHRR:limits-of-abstraction} is the construction of connected layer complexes showing that
\[
H_\clc(4d,d)\ge \Omega(d^2/\log d),
\]
which is not far from the upper bound $(4d-1)d$ in Conjecture~\ref{conj:haehnle}. 

Similarly, in the polymath 3 project it was shown that
\[
2n-O(\sqrt n) \le H_\clc(n,2)\le H_\clm(n,2) =2n-2.
\]
These inequalities are interesting for two reasons. On the one hand, they point again into the direction of  $H_\clm(n,d)$ and $H_\clc(n,d)$ not being too different. But they also highlight the fact that $H_\clm(n,d)$ is more tractable than $H_\clc(n,d)$. We know the exact value of $H_\clm(n,2)$ but not that of $H_\clc(n,2)$, despite quite some effort devoted  in the polymath 3 project to this very specific question.

\subsubsection*{Specific values for small $n$ or $d$}

Not much is known on this besides Corollary~\ref{coro:smalln-d}. In particular, for $H_\clm(n,3)$ we only know
\[
3n-3 \le H_\clm(n,3) \le 4n-4.
\]
The lower bound comes from Examples~\ref{exm:injective-clm} and~\ref{exm:complete-clm}, while the upper bound comes from Theorem~\ref{thm:bl}.
Some effort was devoted (with no success) to deciding which of the two bounds is closer to the truth, since the answer would be an indication of whether $H_\clm(n,3)$ behaves polynomially or not. Of course, the lower bound is simply the value predicted by Conjecture~\ref{conj:haehnle}.

One variation considered in the polymath 3 project, in the hope that it could be simpler, was to drop the restriction that the sets used in connected layer complexes all have the same cardinality. That is, change Definition~\ref{defi:clm} to allow $M$ to be just any family of subsets of $[n]$. (This would not make sense for multisets: unless we pose a bound on the cardinality of the multisets, we can have an infinite number of layers even with $n=1$). Let us denote $H_{\operatorname{np}}(n)$ (for ``non-pure'') the maximum length obtained with them. The following statement summarizes our knowledge about $f(n)$:

\begin{theorem}[polymath 3~\cite{Kalai:polymath3}]
\label{thm:non-pure}
\begin{enumerate}
\item $H_{\operatorname{np}}(n)\le n^ {\log_2n + 1}$ (Kalai-Kleitman bound).
\item $H_{\operatorname{np}}(n)\ge 2n$ for all $n$. In fact, $H_{\operatorname{np}}(n+1)\ge H_{\operatorname{np}}(n)+2$. (Remark: the empty set is allowed to be used as a subset, so $H_{\operatorname{np}}(1)=2$ comes, for example, from $M=\{\emptyset,\{1\}\}$).
\item $H_{\operatorname{np}}(n)=2n$ for $n\le 4$.
\item $H_{\operatorname{np}}(5)\in[11,12]$.
\end{enumerate}
\end{theorem}

\begin{proof}
The proof of part (1) is exactly the same as in Theorem~\ref{thm:kk} (using, in particular, the recursion of Lemma~\ref{lemma:kk}). 
For part (2) consider a connected layer family of maximal length. Without loss of generality, the last layer of it consists of just the empty set. Let $X$ be one of the sets appearing in the previous to last layer. Then between these two layers the following two can be added, when a new element $n+1$ is introduced: a layer containing only $X\cup\{n+1\}$ and a layer containing only $\{n+1\}$.
The upper bounds in parts (3) and (4) use more and more complicated case-studies as $n$ grows, and we skip them. The lower bound for $H_{\operatorname{np}}(5)$  comes from the following example:
\[
1 \quad
15 \quad
\begin{tabular}{c}14\\ 5\end{tabular} \quad
\begin{tabular}{c}12\\ 35\\ 4\end{tabular} \quad
\begin{tabular}{c}13\\ 25\\ 45\end{tabular} \quad
\begin{tabular}{c}245\\ 3\end{tabular} \quad
\begin{tabular}{c}24\\ 34\end{tabular} \quad
234 \quad
23 \quad
2\quad
\emptyset
\]
\qed
\end{proof}

\subsubsection*{Remembering only the support of layers}
One feature of the proof of the quasi-polynomial 
upper bound for $H_\clm(n,d)$ (Theorem~\ref{thm:kk}) is that it works if we know only the \emph{support} of each layer, meaning by this the union of the facets in each layer, rather than the facets themselves.

This suggests an axiomatics for supports, rather than multicomplexes. The following definition (recursive on $n$ and $N$) was proposed in the polymath 3 project:

\begin{definition}
We say that a sequence $\Gamma=\{ S_1,\dots,S_N\}$ of subsets of a set $V$ with $n$ elements is \emph{legal} if it satisfies the following axioms:
\begin{enumerate}
\item[0.] The only legal sequence on $0$ elements is $\{\emptyset\}$.
\item[1.] $\Gamma$ is convex, meaning that $S_i\cap S_k\subset S_j$ for all $i<j<k$.
\item[2.] Every proper subsequence of $\Gamma$ is legal.
\item[3.] If an element $a\in V$ is not used at all in $\Gamma$, then $\Gamma$ is a legal sequence on $n-1$ elements.

\item[4.] If an element $a$ belongs to every $S_i$ then there are subsets $S_i'\subset S_i\setminus \{a\}$ such that $\{S_1',\dots,S_N'\}$ is a legal sequence on $n-1$ elements.
\end{enumerate}
\end{definition}

We denote by $y(n)$ the maximum length of legal sequences of subsets of $[n]$. (Remark: contrary to previous settings, here ``length'' means ``number of layers'' rather than ``number of layers minus one''. This slight inconsistency does not affect the asymptotics of $y(n)$ and makes the following proofs simpler).

\begin{theorem}[Kalai-Kleitman bound, polymath 3~\cite{Kalai:polymath3}]
\label{thm:kk-y}
\[
y(n)\le n^{\frac{\log_2 n}{2}}.
\]
\end{theorem}

Unfortunately, it was soon proved that this new axiomatization was \emph{too general} and that the function $y(n)$ is not polynomial:

\begin{lemma}
For every $n$ and $i$,
\[
y(2(n+i))\ge (i+1) y(n).
\]
\end{lemma}

\begin{proof}
For $i=1$,  observe that the sequence with $y(n)$ copies of a set $A$ of size $n+1$ is valid on $n+1$ elements. Taking two disjoint sets $A$ and $B$ of the same size, the sequence
\[
(A,\dots,A,B,\dots,B),
\]
where each block has length $y(n)$, shows $y(2n+2)\ge 2 y(n)$.

The same idea shows (by induction on $i$) the general statement: sets  if $A$ and $B$ are disjoint with $n+i$ elements, then the sequence
\[
(A,\dots,A,A\cup B,\dots,A\cup B,B,\dots,B),
\]
where the blocks of $A$'s and $B$'s have length $y(n)$ and the block of $A\cup B$'s has length $(i-1)y(n)$ is legal.
\qed
\end{proof}

\begin{corollary}[polymath 3~\cite{Kalai:polymath3}]
\begin{enumerate}
\item $y(4n)\ge n y(n)$ for all $n$.
\item $y(4^k) \ge 4^{k\choose 2}$.
\end{enumerate}
\end{corollary}

\begin{proof}
For part (1), let $i=n$ in the lemma. For part (2), use part (1) and induction on $k$.
\qed
\end{proof}

Observe that this is not far from the upper bound of Theorem~\ref{thm:kk-y}, which specializes to:
\[
y(4^n)\le 4^{k^2}.
\]

\subsubsection*{Reformulations and other abstractions}
The connected layer complexes and multicomplexes that we have been considering in this section are certainly not the first attempt at proving the Hirsch conjecture (or a polynomial version of it) by generalizing and abstracting the properties of graphs of polytopes. Other classical attempts are for example in~\cite{AdlDan:abstract,Kalai:upper-bounds}. (The second one can be considered a prequel to the Kalai-Kleitman quasi-polynomial upper bound). In fact, these earlier attempts were an inspiration for the work of Eisenbrand et al.~in~\cite{EHRR:limits-of-abstraction}.

Taking this into account, Kim~\cite{Kim:abstractions}  started  a study of 
variations of the axioms defining connected layer complexes. His main generalization is that instead of the layers forming a sequence, he allows for them to be attached to the vertices of an arbitrary graph $G$, whose diameter we want to bound.
The \emph{connectedness} condition posed for \clm's or \clc's (condition (1) of Definition~\ref{defi:clm}) becomes:
\begin{equation}
\text{$\forall S\subset[n]$, the star of $S$ intersects a connected subgraph of layers.}
\label{eq:clm-g}
\end{equation}

Apart of this, Kim considers also the following properties, that can be posed for these objects:
\begin{itemize}
\item Adjacency: if two facets of $M$ differ by a single element then they must lie in the same or adjacent layers. (That is, the map from $\graph(M)$ to $G$ associating each facet to its layer is simplicial).
\item Strong adjacency: adjacency holds and every two adjacent layers contain facets differing by an element.
\item Pseudo-manifold, or end-point count: no codimension one face of $M$ is contained in more than two facets of $M$.
\end{itemize}

Among his results are a generalization of the Kalai-Kleitman bound for all layer complexes satisfying connectivity, and some examples showing that without connectivity exponential diameters (of the graph $G$) can occur. For example, diameter $n^{d/4}$ can be obtained for layer complexes that have the end-point count and the strong-adjacency properties. Similar examples were later obtained by H\"ahnle~\cite{Haehnle:abstractions}.

\section{Recent results on the diameter of polytopes and polyhedra}

\subsection{Exact bounds for small $d$ and $n$}
\label{sec:smalln}
Remember that we denote by $H_p(n,d)$ the maximum diameter among all $d$-polyhedra with $n$ facets. The version restricted to bounded polytopes will be denoted $H_b(n,d)$. It is easy to show that $H_p(n,d)\ge H_b(n,d)$ for all $n>d$ and that $H_p(n,d)\ge n-d$ for all $n$ and $d$. The latter is not true for $H_b(n,d)$ (e.g., it is clear that $H_b(n,2)=\lfloor n/2 \rfloor$). 

In this section we report on what exact values of $H_b(n,d)$ are known. 
The recent part are the papers~\cite{BDHS:more-bounds,BreSch:diameter-few-facets}, but let us start with a bit of history. 
The Hirsch Conjecture was soon proved to hold for polytopes and polyhedra of dimension $3$ (we here state only the version for polytopes). A proof can be found in~\cite{KimSan:update}:

\begin{theorem} [Klee~\cite{Klee:paths}]
\label{thm:hirschford3}
$H_b(n,3)= \lfloor \frac{2n}{3} \rfloor - 1$.
\end{theorem}

In fact, the lower bound in this statement is easy to generalize.
Observe that the formula below gives the exact value of $H_b(n,d)$ for $d=2$ as well. Proofs of Theorem~\ref{thm:hirschford3} and Proposition~\ref{prop:lowerhirsch} can be found in~\cite{KimSan:companion}.

\begin{proposition}
\label{prop:lowerhirsch}
\[
H_b(n,d) \ge \left \lfloor\frac{d-1}{d} n \right\rfloor - (d-2).
\]
\end{proposition}



On the other extreme, when $n-d$, rather than $d$, is small, there is also the following classical and important statement of Klee and Walkup. A proof can be found in~\cite[Theorem 3.2]{KimSan:update}.

\begin{theorem}[Klee-Walkup~\cite{KleWal:dstep}]
\label{thm:dstep}
Fix a positive integer $k$. Then
$\max_d H(d+k,d) = H(2d,d).$
\end{theorem}

Here we write $H(n,d)$ without a subscript because the statement is valid for bounded polytopes, for unbounded polyhedra, and for much more general objects (e.g., normal $(d-1)$-spheres or balls). 
In the same paper, Klee and Walkup showed that $H_b(10,5)=5$. With this they concluded the Hirsch Conjecture for polytopes with $n-d\le 5$.
Goodey (1972) computed $H_b(10,4)=5$  and $H_b(11,5)=6$ but was not able to certify that $H_b(12,6)=6$. This was done only recently by Bremner and Schewe~\cite{BreSch:diameter-few-facets}:

\begin{corollary}[Bremner-Schewe~\cite{BreSch:diameter-few-facets}]
The Hirsch bound holds for polytopes with at most six facets more than their dimension.
\end{corollary}

The work from~\cite{BreSch:diameter-few-facets} was later extended by the same authors together with Deza and Hua, which computed $H_b(12,4)=H_b(12,5)=7$. 
All in all,
the following statement exhausts all pairs $(n,d)$ for which the maximum diameter $H_b(n,d)$ of $d$-polytopes with $n$ facets is known. We omit the cases $n < 2d$, because  $H_b(d+k,d)=H_b(2k, k)$ for all $k<d$ (see Theorem~\ref{thm:dstep}), and the trivial case $d\le 2$. 

\begin{theorem} 
\label{thm:hirschforsmalln}
\begin{itemize}
\item $H_b(8,4)=4$ (Klee~\cite{Klee:paths}).
\item $H_b(9,4)=H_b(10,5)=5$ (Klee-Walkup~\cite{KleWal:dstep}).
\item  $H_b(10,4)=5$, $H_b(11,5)=6$ (Goodey~\cite{Goodey}).
\item $H_b(11,4)=H_b(12,6)=6$ (Bremner-Schewe~\cite{BreSch:diameter-few-facets}).
\item $H_b(12,4)=H_b(12,5)=7$ (Bremner-Deza-Hua-Schewe~\cite{BDHS:more-bounds}).
\end{itemize}
\end{theorem}

The results in~\cite{BDHS:more-bounds,BreSch:diameter-few-facets} involve heavy use of computer power. But it would be unfair to say that they are obtained by ``brute force''. In fact, brute force enumeration of the combinatorial types of polytopes with $12$ facets is beyond today's possibilities. 

Instead of that the authors look, for each pair $(n,d)$ under study, at the possible \emph{path complexes}, where a path complex is a simplicial complex of dimension $d-1$ with $n$ vertices and with certain axiomatic restrictions that are necessary for it to be a subcomplex of a simplicial polytope. After enumerating path complexes, the authors address the question of which of them can be completed to be part of the boundary of an actual polytope without the completion producing ``shortcuts". This is modeled in oriented matroid terms: for the path complex to be part of a polytope boundary, certain signs are needed in the chirotope of the polytope. 
The question of whether or not a chirotope with those sign constraints exists is solved with a standard SAT solver, using some ad-hoc decompositions in the instances that are too large for the solver to decide. 
The reduction to satisfiability follows ideas of Schewe~\cite{Schewe:satisfiability} developed originally for a different realizability question.

\subsection{Polytope counter-examples to the Hirsch Conjecture}
\label{sec:Hirsch-counter}

The Hirsch Conjecture for unbounded polyhedra was disproved in 1967 by Klee and Walkup, who showed that:

\begin{theorem}[Klee-Walkup~\cite{KleWal:dstep}]
There is a $4$-polyhedron with 5 facets and diameter five.
\end{theorem}

See~\cite{KimSan:companion,KimSan:update} for a relatively simple description of the Klee-Walkup polyhedron, and its relation to a Hirsch-sharp $4$-polytope with nine facets and to the disproof by Todd~\cite {Todd:MonotonicBoundedHirsch} of the ``monotone Hirsch Conjecture''.
The parameters $n$ and $d$ in this example are smallest possible, as Klee and Walkup also showed. Moreover, it was later shown by Altshuler, Bokowski and Steinberg~\cite{AlBoSt:3spheres} that the Klee-Walkup non-Hirsch polyhedron is unique among simple polyhedra with that dimension and number of facets: every $4$-polyhedron with $8$ facets not combinatorially equivalent to it satisfies the Hirsch bound.

But the same paper of Klee and Walkup contains the more relevant (in our opinion) result that we stated as Theorem~\ref{thm:dstep}: that the Hirsch Conjecture (be it for polytopes, polyhedra, or simplicial balls and spheres) is equivalent to the special case $n=2d$. This special case was dubbed \emph{the $d$-step Conjecture} since it states that we can go from any vertex to any other vertex in (at most) $d$-steps. This result of Klee and Walkup was based in the following easy lemma (see, e.g., \cite[Lemma 3.1]{KimSan:update}):

\begin{lemma}[$d$-step Lemma, Klee-Walkup~\cite{KleWal:dstep}]
Let $P$ be a polyhedron with $n$ facets, dimension $d$, and a certain diameter $\delta$. Then, by a wedge on any facet of $P$, we obtain another polyhedron $P'$ of dimension $d+1$, with $n+1$ facets and with diameter (at least) $\delta$.
\end{lemma}

The first ingredient in the construction of bounded counter-examples to the Hirsch Conjecture is a stronger version of the $d$-step lemma for a particular class of polytopes. We call a polytope a \emph{spindle} if it has two specified vertices $u$ and $v$ such that every facet contains exactly one of them. Put differently, a spindle is the intersection of two pointed cones with apices at $u$ and $v$ (see Figure~\ref{fig:spindle}). The \emph{length} of a spindle is the graph distance between $u$ and $v$.
Spindles that are simple at $u$ and $v$ coincide with what are classically called \emph{Dantzig figures}.
Our statement is, however, about spindles that are \emph{not} simple:

\begin{figure}[htb]
\begin{center}
\resizebox{0.6\linewidth}{!}{
  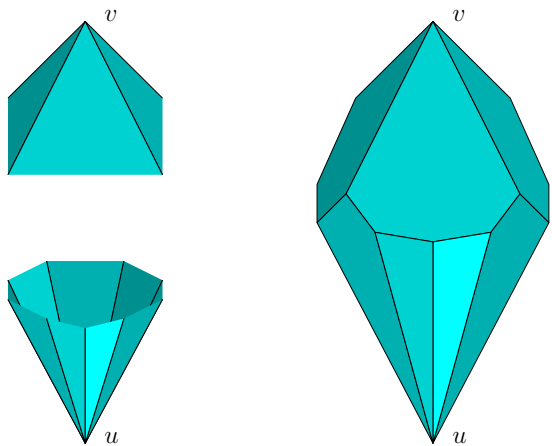
}
\caption{A spindle. 
}
\label{fig:spindle}
\end{center}
\end{figure}

\begin{theorem}[strong $d$-step Lemma, Santos~\cite{Santos:Hirsch-counter}]
\label{thm:spindle}
Let $P$ be a spindle with $n$ facets, dimension $d$, and a certain length $l$, and suppose $n>2d$. Then, by a wedge on a certain facet of $P$ followed by a perturbation, we can obtain another spindle $P'$ of dimension $d+1$, with $n+1$ facets and with length (at least) $l+1$.
\end{theorem}

Observe that this statement is stronger than the classical $d$-step lemma only in the ``$+1$'' in $l$. This is enough, however, to obtain the following crucial corollary:

\begin{corollary}
\label{coro:spindle}
If a spindle $P$ has length greater than its dimension, then applying  Theorem~\ref{thm:spindle} $n-2d$ times we obtain a polyhedron violating the Hirsch Conjecture. More precisely, from a $d$-spindle with $n>2d$ facets and length $l>d$ we obtain a polyhedron of dimension $D=n-d$, with $N=2(n-d)$ facets, and diameter at least $l +(n-2d) >D=N-D$.
\qed
\end{corollary}

What Santos~\cite{Santos:Hirsch-counter}  then does is he constructs a spindle with dimension $5$ and length $6$. Let us mention that he does this (and proves Theorem~\ref{thm:spindle}) in a dual setting in which the Hirsch question is about the dual diameter of a polytope, similar to the one used in Section~\ref{sec:complexes}. The duals of spindles are called \emph{prismatoids}.

In the original paper~\cite{Santos:Hirsch-counter}, a $5$-spindle of length $6$ with $48$ facets is constructed. 
In a subsequent paper of Santos with Matschke and Weibel~\cite{MaSaWe:5prismatoids} the number of facets needed to construct a $5$-spindle (or, vertices to construct a $5$-prismatoid) of length $6$ is reduced to $25$. The number of facets is irrelevant for the main conclusion (the existence of a non-Hirsch polytope) but is important for the size of it. In particular, the non-Hirsch polytopes of \cite{Santos:Hirsch-counter} and \cite{MaSaWe:5prismatoids} have, respectively, dimensions $43$ and $20$, the latter being the smallest dimension in which we now know the bounded Hirsch Conjecture to fail. This difference in size and complexity also affects how explicit the example is. The original one is too big to be computed completely (it requires $38$ iterations of Theorem~\ref{thm:spindle}) while the new one has been explicitly computed by Weibel. See details in~\cite{MaSaWe:5prismatoids}. Another interesting result of \cite{MaSaWe:5prismatoids} is that in dimension $5$ there are arbitrarily long spindles:

\begin{theorem}[Matschke-Santos-Weibel~\cite{MaSaWe:5prismatoids}]
\label{thmPrismatoidWithArbitraryLargeWidth-intro}
For every $k$ there is
a $5$-dimensional spindle with $12k(6k-1)$ facets and of width $4+k$.
\end{theorem}

Why start with spindles of dimension $5$? It is easy to show that the length of a $3$-spindle is at most three. Regarding $4$-spindles, it was left as an open question in the first version of~\cite{Santos:Hirsch-counter} whether they can have length greater than $4$, but the question was soon answered in the negative by Santos, Stephen and Thomas~\cite{SaStTh:4prismatoids}. Quite remarkably, the proof of this statement is just topological graph theory in the most classical sense of the expression; that is, the study of graph embeddings in closed surfaces. What~\cite{SaStTh:4prismatoids} shows is that when two arbitrary graphs $G_1$ and $G_2$ are embedded (transversally to one another) in the $2$-sphere, it is always possible to go from a vertex of $G_1$ to  vertex of $G_2$ via a single crossing of an edge from $G_1$ and $G_2$. This is not true for other surfaces (e.g., the torus). More significantly, the fact that this is not true in the torus is at the heart of all the constructions of $5$-spindles of length greater than $5$, via the standard Clifford embedding of the torus in the $3$-sphere and a reduction of the combinatorics of $5$-spindles to the study of certain cell decompositions of the $3$-sphere.

In the rest of this section we study how good (or how bad) the counter-examples to the Hirsch Conjecture that can be obtained with this spindle method are. Following~\cite{Santos:Hirsch-counter} we call \emph{(Hirsch) excess} of a $d$-polytope $P$ with $n$ facets and diameter $\delta$ the quantity
\[
\frac{\delta}{n-d} -1.
\]
This is positive if and only if the diameter of $P$ exceeds the Hirsch bound. Excess is a significant parameter since, as shown in~\cite[Section 6]{Santos:Hirsch-counter}, from any non-Hirsch polytope one can obtain infinite families of them with (essentially) the same excess as the original, even in fixed (but high) dimension:

\begin{theorem}[Santos~\cite{Santos:Hirsch-counter}]
\label{thm:asymptotic2}
Let $P$ be a non-Hirsch polytope of dimension $d$ and excess $\epsilon$. Then for each $k\in\N$ there is an infinite family of non-Hirsch polytopes of dimension $kd$ and with excess greater than
\[
\left(1-\frac{1}{k}\right)\epsilon.
\]
\end{theorem}

The excess of the non-Hirsch polytope produced via Theorem~\ref{thm:spindle} from a $d$-spindle $P$ of length $l$ and $n$ vertices equals
\[
\frac{l-d}{n-d},
\]
so we call this quotient the \emph{(spindle) excess} of $P$. The spindle of~\cite{MaSaWe:5prismatoids}, hence also the non-Hirsch polytope obtained from it, has excess $1/20$. This is the greatest excess of a spindle or polytope constructed so far. (The excess of the Klee-Walkup \emph{unbounded} non-Hirsch polyhedron is, however, $1/4$). 

It could seem that the arbitrarily long spindles mentioned in Theorem~\ref{thmPrismatoidWithArbitraryLargeWidth-intro} should lead to non-Hirsch polytopes of greater excess. This is not the case, however, because their number of facets  grows quadratically with the length. So, asymptotically, their spindle excess is zero.

On the side of upper bounds, the Barnette-Larman general bound for the diameters of polytopes~\cite{Larman:upper-bound} (see the version for connected layer families in Theorem~\ref{thm:bl}) implies that the excess of spindles of dimension $d$ cannot exceed $2^{d-2} /3$. For dimension $5$, this is improved to $1/3$ in~\cite{MaSaWe:5prismatoids}.
What are the implications of this? Of course, using spindles of higher dimension it may be possible to get non-Hirsch polytopes of great excess. But the main point of the spindle method is that it relates non-Hirschness to length and dimension alone, regardless of number of facets or vertices, which gives a lot of freedom to construct complicated spindles in fixed dimension. It seems to us that giving up ``fixed dimension''  in this method leads to a problem as complicated as the original Hirsch question. On the other hand, Barnette-Larman's bound implies that the method in fixed dimension cannot produce non-Hirsch polytopes whose excess is more than a constant. That is, the method does not seem to be good enough to disprove a ``linear Hirsch Conjecture'' (in case it is false).

\subsection{The Hirsch Conjecture holds for flag normal complexes}
\label{sec:flag-complexes}

A simplicial complex $C$ is a \emph{flag complex} (or a \emph{clique complex}) if, whenever the graph of $C$ contains the complete graph on a certain subset $S$ of vertices, we have that $S$ is a face in $C$. (Here we mean the usual graph of $C$, not the adjacency or dual graph). Equivalently, $C$ is flag if every \emph{minimal non-face} (that is, every inclusion minimal subset of vertices that is not a face) has cardinality two. 
Examples of flag complexes include the barycentric subdivisions of arbitrary simplicial complexes.

In this section we reproduce Adiprasito and Benedetti's recent proof of the Hirsch bound for simple polytopes whose polar simplicial complex is flag~\cite{AdiBen:flag-complexes}. 
The proof actually works not only for polytopes, but for all \emph{normal} (or \emph{locally strongly connected}, in the terminology of~\cite{JosIzm:branched-coverings}) and flag complexes. That is to say:

\begin{theorem}[Adiprasito-Benedetti~\cite{AdiBen:flag-complexes}]
\label{thm:hirsch-for-flag}
Let $C$ be a normal and flag pure simplicial complex of dimension $d-1$ on $n$ vertices. Then the adjacency graph of $C$ has diameter at most $n-d$.
\end{theorem}

The proof is via \emph{non-revisiting paths}. A path $X_0,X_1,\dots, X_N$ in the dual graph of a pure complex $C$ is called \emph{non-revisiting} if its intersection with the star of every vertex is connected. Put differently, if for every $0\le i<j<k\le N$ we have $X_i\cap X_k\subset X_j$. Observe the similarity with the connectivity condition for connected layer families in Section~\ref{sec:clm}. The relation of non-revisiting paths to the Hirsch Conjecture is:

\begin{proposition}
\label{prop:nonrevisiting}
A non-revisiting path in a pure $(d-1)$-complex with $n$ vertices cannot have length greater than $n-d$.
\end{proposition}

\begin{proof}
In a dual path $X_0,X_1,\dots,X_N$, at every step from $X_i$ to $X_{i+1}$ a unique vertex is abandoned and a unique vertex is introduced. The non-revisiting condition says that no vertex can be introduced twice, and that the $d$ initial vertices in $X_0$ cannot be (abandoned and then) reintroduced. That is, the number of vertices introduced in the $N$ steps, hence the number of steps, is bounded above by $n-d$.
\qed
\end{proof}

In fact, Klee and Walkup~\cite{KleWal:dstep} showed that the Hirsch Conjecture for polytopes was equivalent to the conjecture that every pair of facets in a simplicial polytope could be joined by a non-revisiting dual path. The latter was the \emph{non-revisiting path conjecture} posed earlier~\cite{Klee:paths} and attributed to Klee and Wolfe. What Adiprasito and Benedetti prove (which implies Theorem~\ref{thm:hirsch-for-flag}) is:

\begin{theorem}[Adiprasito-Benedetti~\cite{AdiBen:flag-complexes}]
\label{thm:nonrevisiting-for-flag}
 Every pair of facets in a normal and flag pure simplicial complex can be joined by a non-revisiting dual path.
\end{theorem}

Adiprasito and Benedetti give two proofs of Theorem~\ref{thm:nonrevisiting-for-flag}. We will concentrate in their \emph{combinatorial proof}, but 
it is also worth sketching their \emph{geometric proof}.

\begin{proof}[Geometric proof of Theorem~\ref{thm:nonrevisiting-for-flag}, sketch]
We introduce the following metric in $C$: map each simplex to an orthant of the unit $(d-1)$-sphere, so that every edge has length $\pi/2$, every triangle is spherical with angles $\pi/2$, etcetera. Glue the metrics so obtained in adjacent facets via the unique isometry that preserves vertices. This is called  the \emph{right angled metric} on $C$. 

Gromov~\cite{Gromov:groups} proved that in the right angled metric of a flag complex the (open) star of every face is geodesically convex. 
That is, the shortest path between two points in the star is unique and stays within the star. This implies that the minimum geodesic between respective interior points in two given facets cannot revisit the star of any vertex. Hence, if the geodesic induces a dual path, this path is non-revisiting.

The reason why this is only a ``sketch'' of proof, and the reason why normality of $C$ is needed, is that such a geodesic may not induce a dual path, if it goes through faces of codimension higher than one. More critically, in some complexes \emph{every} geodesic between interior points of two given facets of may need to go through faces of codimension higher than one. (Consider for example the complex $C$ obtained as the cone over a path of length at least four. Every geodesic between a point in the first triangle and a point in the last triangle necessarily goes through the apex, hence it does not directly induce a path in the adjacency graph of $C$). To solve this issue a perturbation argument needs to be used: if the geodesic crosses a face $F$ of codimension greater than one, a short piece of the geodesic around $F$ is modified, using recursion on the link of $F$. In particular, all links need to be strongly connected; that is, the complex $C$ needs to be normal.
\qed
\end{proof}

The combinatorial proof cleverly mixes two notions of distance between facets in the complex $C$. Apart of the dual graph distance (which is our ultimate object of study) there is the \emph{vertex distance}, in the following sense:

\begin{definition}
\label{defi:vertexdistance}
Let $S$ and $T$ be two subsets of vertices of a simplicial complex $C$ (for example, but not necessarily, the vertex sets of two faces of $C$). The \emph{vertex-distance} between $S$ and $T$ in $C$, denoted $\vdist_C(S,T)$, is the minimum distance, along the graph of $C$, between a vertex of $S$ and a vertex of $T$.
\end{definition}

Of course, the vertex distance between $S$ and $T$ is zero if and only if $S\cap T\ne\emptyset$.  Adiprasito and Benedetti introduce a particular class of dual paths in $C$ that they call \emph{combinatorial segments}. The definition is subtle in (at least) two ways:
\begin{itemize}
\item It uses double recursion on the dimension of $C$ and on the vertex-distance between the end-points.
\item It is \emph{asymmetric}. Combinatorial segment do not join two facets or two vertices, but rather go \emph{form a facet $X$ to a set of vertices $S$}. Eventually, we will make $S$ to be a facet, but it is important in the proof to allow for more general sets. 
\end{itemize}

\begin{definition}
\label{defi:segment}
Let $X\in C$ be a facet and $S\subset V$ be a set of vertices in a pure and normal simplicial $(d-1)$-complex $C$ with vertex set $V$. 
Let $x\in X$. We say that a facet path $(X=X_0,X_1,\dots,X_N)$ is a \emph{combinatorial segment from $X$ to $S$ anchored at $x$} if either:
\begin{itemize}
\item $S\cap X\ne \emptyset$ (in particular, $x\in X\cap S$) and $N=0$. That is, the path is just $(X)$. (Distance zero).
\item $d=1$, $S\cap X=\emptyset$ and $N=1$, so that the path is $(X,\{v\})$ for a $v\in S$. (Dimension zero).
\item $d>1$, $S\cap X=\emptyset$ and the following holds. 
\begin{enumerate}
\item $X_N$ is the unique facet in the path intersecting $S$.
\item Let $\delta=\vdist_C(X,S)$. Let $X_k$ be the first facet in $\Gamma$ with $\vdist_C(X_k,S)<\delta$ and let $y\in X_k$ be the unique
vertex with $\vdist_C(y,S)=\vdist_C(X_k,S)=\delta-1$. (That is, let $y$ be the only element of $X_k\setminus X_{k-1}$).
Then $x\in X_0\cap X_1\cap \dots \cap X_k$ and  the link of $x$ in $\Gamma_1:=(X_0,\dots,X_k)$ is a combinatorial segment in $\lk_C(x)$
from the facet $X\setminus x$ to the set $\{z\in V: \vdist_C(z,x)=1, \vdist_C(z,S)=\delta-1\}$.
\item $\Gamma_2:=(X_k,\dots,X_N)$ is a combinatorial segment from $X_k$ to $S$ in $C$ anchored at $y$.
\end{enumerate}
\end{itemize}
\end{definition}

Some immediate consequences of this definition are:

\begin{lemma} In the conditions of Definition~\ref{defi:segment}:
\label{lemma:segment}
\begin{enumerate}
\item $\vdist_C(X,S)=\vdist_C(x,S)$. 
\item The distance $\vdist_C(X_l,S)$ is weakly decreasing along the whole path.
\item  $\forall l\in\{1,\dots,N\}$,  $(X_l,\dots,X_N)$ is a combinatorial segment from $X_l$ to $S$ (and, if $l< k$, the segment is anchored at $x$). 
\end{enumerate}
\end{lemma}

\begin{proof}
For (1), since $x\in X\cap X_k$ we have $\vdist(X,S) \le \vdist(x,S)\le \vdist(X_k,S) +1<\vdist(X,S)+1$. 
(2) is clear for $\Gamma_1$ and true in $\Gamma_2$ by induction on $\vdist(X,S)$. 
(3) is again trivial by induction on the dimension and the distance $\vdist(X,S)$.
\qed
\end{proof}

The existence of combinatorial segments may not be obvious, so let us prove it. In this and the following proofs when we say ``induction on the dimension and distance'' we mean that the statement is assumed for all complexes of smaller dimension, and for all complexes of the same dimension and pairs $(X',S')$ at smaller distance.

\begin{proposition}
\label{prop:segment}
Let $X$ be a facet and $S\subset V$ be a set of vertices in a normal pure simplicial complex $C$. Let $x\in X$ be with $\vdist_C(X,S)=\vdist_C(x,S)$. Then there exists a combinatorial segment from $X$ to $S$ anchored at $x$.
\end{proposition}

\begin{proof}
We use induction on the dimension and distance, the cases of dimension or distance zero being trivial.

Suppose then that $d>1$ and that $\vdist(S, X)=\delta>0$. By induction on dimension, there is a  a combinatorial segment $\Gamma'$ in $\lk_C(x)$ from the facet $X\setminus x$ to the set $\{z\in V: \vdist_C(z,x)=1, \vdist_C(z,S)=\delta-1\}$ (we do not care about the anchor of $\Gamma'$). Let $\Gamma_1=(X_0,X_1,\dots,X_k)$ be the join of $\Gamma'$ and $x$, let $y$ be the vertex in $X_k\setminus X_{k-1}$ and let $\Gamma_2$ be any combinatorial segment from $X_k$ to $S$ in $C$, which exists by induction on distance.
\qed
\end{proof}

The following property of combinatorial segments is crucial for the proof of Theorem~\ref{thm:nonrevisiting-for-flag}:

\begin{lemma}
\label{lemma:combinatorial-segments}
Let $\Gamma$ be a combinatorial segment from a facet $X$ to a set $S$ in a positive-dimensional flag and normal pure complex $C$. Let $k$, $\Gamma_1=(X,X_1,\dots,X_k)$, $x$ and $y$ be as in Definition~\ref{defi:segment}.
Then for every $z$ with $\vdist(z,y)=1$ that is used in some $X_l$ of $\Gamma_1$ we have $z\in X_l\cap X_{l+1} \cap \dots \cap X_k$.
\end{lemma}

\begin{proof}
Without loss of generality (by property (3) of Lemma~\ref{lemma:segment}) we can assume $l=0$ or, put differently, $z\in X_0$.
Also, we can assume $z\ne x$ since for $z=x$ there is nothing to prove.

Observe that $xy$, $xz$ and $yz$ are edges in $C$. Since $C$ is flag, $xyz$ is a face. In particular, the lemma is trivially true (or, rather, void) for the case of dimension one ($d=2$). For the rest we assume $d\ge 3$ and use induction on $d$.

Consider the combinatorial segment $\Gamma':=\lk_{\Gamma_1}(x)=(X'_0,X'_1,\dots,X'_k)$ in $\lk_C(x)$. Since $\vdist_{\lk_C(x)}(X',y)=1$ (because $z\in X'_0$) the decomposition of $\Gamma'$ into a $\Gamma'_1$ and a $\Gamma'_2$ is just $\Gamma'_1=\Gamma'$ and 
$\Gamma'_2=(X'_k)$. Inductive hypothesis implies that $z\in X'\cap X'_1\cap\dots\cap X'_k$.
\qed
\end{proof}

\begin{corollary}
Combinatorial segments in flag normal complexes are non-revisiting.
\end{corollary}

\begin{proof}
We prove this by induction on dimension and distance, the cases of dimension or distance zero being trivial.

In case $\vdist_C(X,S)\ge 0$ and $d>1$, let $\Gamma_1=(X=X_0,\dots,X_k)$, $\Gamma_2$, $x$ and $y$ be
as in Definition~\ref{defi:segment}. Since $\Gamma_1$ and $\Gamma_2$ are non-revisiting by inductive hypothesis, the only thing to prove is that it is not possible for a vertex $z$ used in $\Gamma_1$ and $\Gamma_2$ not to be in $X_k$. 

Let $\delta=\vdist(X,v)$. Observe that $\vdist(z,S)=\delta$, since $z$ belongs both to facets of $\Gamma_1$ (which are at distance $\delta$ from $S$) and of $\Gamma_2$ (at distance less than $\delta$ from $S$). The proof is now by induction on $\delta$:
\begin{itemize}
\item If $\delta=1$ then $\Gamma_2=(X_k)$ and there is nothing to prove.
\item If $\delta>1$, decompose $\Gamma_2$ as the concatenation of a $\Gamma'_1$ and a $\Gamma'_2$ as in Definition~\ref{defi:segment}. Since $\Gamma_2$ is anchored at $y$,  $\Gamma'_1$ is contained in the star of $y$. Now, $\Gamma'_2$ contains only facets at distance less than $\delta-1$ to $S$ so, in particular, it does not contain $z$. That implies $\vdist_C(z,y)=1$ and, by Lemma~\ref{lemma:combinatorial-segments}, $z\in X_k$.
\end{itemize}
\qed
\end{proof}

\begin{proof}[Combinatorial proof of Theorem~\ref{thm:nonrevisiting-for-flag}]
Let $X$ and $Y$ be two disjoint facets in $C$. (For non-disjoint facets, use induction in the link of a common vertex).

Let $\Gamma=(X,\dots,X_N)$ be a combinatorial segment from $X$ to the vertex set $Y$,
and let $v\in X_N\cap Y$. By induction on the dimension, consider a non-revisiting path $\Gamma'$ from $X_N$ to $Y$ in the star of $v$.
We claim that the concatenation of $\Gamma$ and $\Gamma'$ is non-revisiting. Since both parts are non-revisiting ($\Gamma$ by Lemma~\ref{lemma:combinatorial-segments}), the only thing to prove is that it is not possible for a vertex $z$ used in $\Gamma$ and $\Gamma'$ not to be in $X_N$. 

Such a $z$ must be at distance 1 from $v$ (because $\Gamma'$ is contained in the star of $v$),
so the first facet of $\Gamma$ containing $z$ is at distance 1 from $v$. By property (3) of Lemma~\ref{lemma:segment} there is no loss of generality in assuming that $z\in X$ and that $\vdist(X,Y)=1$. In this case, $\Gamma$ coincides with the path $\Gamma_1$ of Definition~\ref{defi:segment} and $y$ is the vertex in $X_N\setminus X_{N-1}$. By Lemma~\ref{lemma:combinatorial-segments}, $z$ is in $X_N$.
\qed
\end{proof}

\subsection{Highly non-decomposable polyhedra do exist}
\label{sec:decomposable}

In 1980, Provan and Billera~\cite{ProBil:decomposable} introduced the following concepts for simplicial complexes, and proved the following results:

\begin{definition}[\protect{\cite[Definition 2.1]{ProBil:decomposable}}]
\label{defi:decomposable}
Let $C$ be a pure $(d -1)$-dimensional simplicial complex and let $0 \le k \le d - 1$. We say that $C$ is \emph{$k$-decomposable} if either
\begin{enumerate}
\item $C$ is a $(d - 1)$-simplex, or 
\item there exists a face $S\in C$ (called a \emph{shedding face}) with $\dim(S) \le k$
such that
\begin{enumerate}
\item $C\setminus S$ is $(d - 1)$-dimensional and $k$-decomposable, and
\item $\lk_C(S)$ is $(d - |S| - 1)$-dimensional and $k$-decomposable.
\end{enumerate}
\end{enumerate}
\end{definition}

\begin{definition}[\protect{\cite[Definition 4.2.1]{ProBil:decomposable}}]
\label{defi:w-decomposable}
Let $C$ be a pure $(d -1)$-dimensional simplicial complex and let $0 \le k \le d - 1$. We say that $C$ is \emph{weakly $k$-decomposable} if either
\begin{enumerate}
\item $C$ is a $(d - 1)$-simplex, or 
\item there exists a face $S\in C$ (called a \emph{shedding face}) with $\dim(S) \le k$
such that
$C\setminus S$ is $(d - 1)$-dimensional and weakly $k$-decomposable
\end{enumerate}
\end{definition}

In these statements, $C\setminus S$ denotes the simplicial complex obtained removing from $C$ all the facets that contain $S$. This is called the \emph{deletion} or the \emph{antistar} of $S$ in $C$. The main motivation of Provan and Billera, as the title of their paper indicates, is to relate decomposability to the diameter of the adjacency graph of the complex:

\begin{theorem}
\label{thm:provanbillera}
Let $C$ be a pure $(d-1)$-dimensional complex. Let $f_k(C)$ denote the number of faces of dimension $k$ in $C$, for $k=0,\dots,d-1$. Let $\diam(C)$ denote the diameter of the adjacency graph of $C$. Then:
\begin{enumerate}
\item If $C$ is $k$-decomposable then $\diam(C) \le f_k(C) - \binom{d}{k+1}$. 
\item If $C$ is weakly $k$-decomposable then $\diam(C) \le 2f_k(C)$.
\end{enumerate}
\end{theorem}

Observe that for $k=d-1$ definitions~\ref{defi:decomposable} and~\ref{defi:w-decomposable} are equivalent (since the link condition becomes void). In fact, $(d-1)$-decomposable complexes of dimension $d-1$ are the same as \emph{shellable} complexes, and include all the face complexes of simplicial $d$-polytopes. In this sense the concept of $k$-decomposability, for varying $k$, interpolates between the face complexes of all polytopes and those of $0$-decomposable (or \emph{vertex decomposable}) ones, which satisfy the Hirsch bound, by part (1) of Theorem~\ref{thm:provanbillera}.

Simplicial polytopes with non-vertex-decomposable boundary complexes were soon found. Klee and Kleinschmidt in their 1987 survey on the Hirsch conjecture observe that a certain polytope constructed by E.~R.~Lockeberg is not vertex-decomposable~\cite[p.~742]{KleKle:dstep}. But the question remained open whether the same happens for weakly vertex-decomposable, or for $k$-decomposable with higher $k$. The two questions have been solved recently, and with surprisingly simple (in every possible sense of the word) polytopes.

Remember that the $k$-th hypersimplex of dimension $d$ is the intersection of the standard $(d+1)$-cube with the hyperplane $\{\sum x_i=k\}$, for an integer $k\in\{1,\dots,d\}$~\cite{DeRaSa:triang_book,Ziegler:lectures-on-polytopes}. We generalize the definition as follows:

\begin{definition}
\label{defi:hypersimplex}
Let $a$ and $b$ be two positive integers. We call fractional hypersimplex of parameters $(a,b)$ and dimension $a+b$ the intersection
of the standard $(a+b+1)$-cube $[0,1]^{a+b+1}$ with the hyperplane $\{\sum x_i = a+0.5\}$. We denote it $\Delta_{a,b}$.
\end{definition}

Let us list without proof several easy properties of fractional hypersimplices:
\begin{enumerate}
\item $\Delta_{a,b}$ is combinatorially equivalent to $\Delta_{b,a}$, and to the Minkowski sum of the $a$-th and $a+1$-th hypersimplices of dimension $a+b$.
\item $\Delta_{a,b}$ is the $2\times (a+b+1)$ transportation polytope obtained with margins $(a+0.5,b+0.5)$ in the rows and $(1,\dots,1)$ in the columns.
\item $\Delta_{a,b}$ is simple and it has $2a+2b+2$ facets (one for each facet of the $(a+b+1)$-cube.
\item By the previous property, a subset of facets of $\Delta_{a,b}$ can be labeled as a pair $(S,T)$ of subsets of $[a+b+1]$, with an element $i\in S$ representing the facet $\{x_i=0\}$ and an element $j\in T$ representing the facet $\{x_j=1\}$. Then the vertices of $\Delta_{a,b}$ correspond exactly to the pairs $(S,T)$ with $S\cap T=\emptyset$, $|S|=a$, and $|T|=b$.
\item In particular, $\Delta_{a,b}$ has exactly $(a+b+1)\binom{a+b}{a}$ vertices.
\end{enumerate}

The main results concerning decomposability of $\Delta_{a,b}$ are:

\begin{theorem}
\label{thm:indecomposable}
Let $\nabla_{a,b}$ denote the polar of the fractional hypersimplex $\Delta_{a,b}$.
\begin{enumerate}
\item \emph{(De Loera and Klee~\cite{DelKle:decomposable}).}
$\nabla_{a,b}$ is not weakly vertex-decomposable for any $a,b\ge 2$. In particular, $\nabla_{2,2}$ is a non-weakly-vertex decomposable simplicial $4$-polytope with $10$ vertices and $30$ facets.  
\item \emph{(H\"ahnle, Klee and Pilaud~\cite{HaKlPi:decomposable}).}
$\nabla_{a,b}$ is not weakly $k$-decomposable for any $k\le \sqrt{2\min(a,b)} -3$. In particular, for every $k$ there is a non-weakly-$k$-decomposable polytope of dimension $2\left\lceil(k+3)^2/4\right\rceil$ with $(k+3)^2+2$ vertices. 
\end{enumerate}
\end{theorem}

\begin{proof}
We only give the proof of part (1). Part (2) follows similar ideas except the details are trickier.

What De Loera and Klee show is that there cannot be a \emph{shedding sequence} $(i_1,i_2,i_3,\dots)$ simply because either 
$(\nabla_{a,b}\setminus i_1) \setminus i_2$ or $((\nabla_{a,b}\setminus i_1) \setminus i_2)\setminus i_3$ will not be pure, no matter who the vertices $i_1$, $i_2$ and $i_3$ are. 

Remember that each vertex of $\nabla_{a,b}$ (that is, each facet of $\Delta_{a,b}$) corresponds to a facet $\{x_i=0\}$ or $\{x_i=1\}$ of the cube, with $i\in[a+b+1]$. In what follows we label the vertex corresponding to $\{x_i=0\}$ as $+i$ and the vertex corresponding to $\{x_i=1\}$ as $-i$, so that the vertex set of $\nabla_{a,b}$ is $\{\pm1,\pm 2,\dots,\pm(a+b+1)\}$. 
Without loss of generality assume that at least two of the first three vertices $i_1$, $i_2$ and $i_3$ in the shedding sequence are of the ``$+$'' form. There are then two cases:
\begin{itemize}
\item If $i_1$ and $i_2$ are both of the ``$+$'' form, assume without loss of generality that $\{i_1,i_2\}=\{+1,+2\}$. Let $C= (\nabla_{a,b}\setminus i_1) \setminus i_2$.
\item If not, assume without loss of generality that $\{i_1,i_2,i_3\}=\{+1,+2,-1\}$ or  $\{i_1,i_2,i_3\}=\{+1,+2,-3\}$. Let $C= ((\nabla_{a,b}\setminus i_1) \setminus i_2)\setminus i_3$.
\end{itemize}
In all cases $C$ is full-dimensional, since it contains (for example) one of the facets
\[
\{-2,-3,\dots,-(b+1),+(b+2),\dots,+(a+b+1)\}
\]
or
\[
\{-1,-2,-4,\dots,-(b+1),+(b+2),\dots,+(a+b+1)\}.
\]
But $C$ is not pure, since it does not contain any of the following two facets of $\nabla_{a,b}$ but, still, it contains their common ridge:
\begin{eqnarray*}
\{+1,+3,\dots,+(a+1), -(a+2),\dots,-(a+b+1)\}\\
\{+2,+3,\dots,+(a+1), -(a+2),\dots,-(a+b+1)\}.
\end{eqnarray*}
\qed
\end{proof}

\subsection{Polynomial diameter of polyhedra with bounded coefficients}
\label{sec:subdeterminants}

Although the main motivation for the Hirsch question is in the world of true, geometric polyhedra, in most of the paper we have been using combinatorial or topological ideas, even when the results mentioned were specific to realized (or realizable) objects, as was the case in Section~\ref{sec:smalln}. But we finish with an intrinsically geometric issue, the role of the size of coefficients in the diameter of a polyhedron or polytope.
The motivation for studying this is two-fold:
\begin{itemize}
\item Since we cannot find a polynomial upper bound for the diameters of polyhedra in terms of $n$ and $d$ alone, it may be interesting to understand whether we can do it in terms of $n$, $d$ and the size of the coefficients, where ``size'' should be understood as ``bit-length'' (or number of digits). Such a bound would be a step towards a (non-strongly) polynomial-time simplex method. 
\item Perhaps that is too optimistic; but bounding the diameter in terms of the size of coefficients will at least give polynomial upper bounds for the diameters of particular classes of polytopes and polyhedra. 
As a classical example of this, in 1994 Dyer and Frieze~\cite{DyeFri:random-walks} gave a polynomial bound on the diameter of polyhedra whose defining matrix is \emph{totally unimodular}, a case that includes, for example, all network flow polytopes:


\begin{theorem}[Dyer and Frieze~\cite{DyeFri:random-walks}]
Let $A$ be a totally unimodular $n \times d$ matrix and $c$ a vector in $\R^n$.  Then the diameter of the polyhedron
$P = \{{\bf x} \in \R^d : A{\bf x} \leq c \}$
is $O(d^{16}n^3(\log(dn))^3)$.
\end{theorem}

Here, the polyhedron $P$ can be assumed to be a $d$-polytope with $n$ facets. The proof is based on a randomized simplex algorithm. 
\end{itemize}

In the same vein, polynomial upper bounds exist for the diameters of the following classes of polytopes (see~\cite{BDEHN:subdeterminants,KimSan:update} and the references therein): Naddef (1989) proved the Hirsch bound for polytopes whose vertices have only 0/1 coordinates. Orlin (1997) proved a quadratic upper bound for network flow polytopes and Balinski (1984) proved the Hirsch bound for their linear programming duals. Brightwell, van den Heuvel and L.~Stougie (2006), improved by Hurkens (2007) have shown a linear bound for classical transportation polytopes. Their method was then generalized by De Loera, Kim, Onn and Santos (2009) to yield a quadratic upper bound for $3$-way axial transportation polytopes.

But the two most general result in this direction are the following 20-year old one by Kleinschmidt and Onn~\cite{KleOnn:diameter}, and the following very recent one by Bonifas et al.~\cite{BDEHN:subdeterminants}:

\begin{theorem}[Kleinschmidt and Onn~\cite{KleOnn:diameter}]
The diameter of a polytope with all its vertices integer and contained in $[0,k]^d$ cannot exceed $kd$.
\end{theorem}

\begin{theorem}[Bonifas et al.~\cite{BDEHN:subdeterminants}]
\label{thm:subdeterminants}
Let $P=\{x \in \R^d \colon Ax\leq b\}$ be a polytope defined by an integer matrix $A\in \Z^{n\times d}$ and suppose all
subdeterminants of $A$ are bounded in absolute value by a certain  $M\in \N$.
The, the diameter of $P$ is bounded by $O\left(M^2 d^{3.5}\log dM\right)$.
\end{theorem}

What is remarkable about this result is that, when applied to totally unimodular matrices (taking  $M=1$) it gives a much better bound than the original one by Dyer and Frieze. We briefly sketch the main ideas in the proof of Theorem~\ref{thm:subdeterminants}:

\begin{itemize}
\item Without loss of generality $P$ is supposed to be simple. This can be achieved by slightly perturbing the right-hand side $b$, whose coefficients are not taken into account in $M$ or even assumed to be integer. The perturbation can only increase the diameter.

\item The proof then works in the \emph{normal fan} of $P$, which is a decomposition of the (dual) vector space $\R^n$ into simplicial cones 
$c_v$ corresponding to the vertices $v$ of $P$. To each simplicial cone we associate a \emph{spherical volume}, the volume of its intersection with the unit ball. Observe that the cone associated to a vertex $v$ is independent of the right-hand side $b$. What $b$ controls is only which cones appear (that is, which bases of $A$ correspond to vertices of $P$).

\item Bonifas et al.~then fix two cones $c_u$ and $c_v$ and grow \emph{breadth-first-search} trees in the dual graph of the normal fan (that is, in the normal fan of $P$) starting from those cones. Put differently, they consider the cones $U_i$ and $V_i$ for all $i\in \N$ with:
\begin{itemize}
\item $U_0=c_u$, $V_0=c_v$.
\item $U_{i+1}$ equal to $U_i$ together with all the vertex cones adjacent to $U_i$, and the same for $V_{i+1}$.
\end{itemize}

\item The main idea in the proof is then to study how the volume of $U_i$ and $V_i$ grow with $i$. When both can be guaranteed to be bigger than half of the volume of the ball (which is estimated from above as $2^d$) we are sure that $U_i$ and $V_i$ have a common vertex-cone, so the distance from $u$ to $v$ in the graph of $P$ is at most $2i$. 
\end{itemize}

For the last step the following ``volume expansion'' result is crucial:

\begin{lemma}
\label{lemma:volume-expansion}
Let $P=\{x \in \R^d \colon Ax\leq b\}$ be a polytope defined by an integer matrix $A\in \Z^{n\times d}$ and suppose all
subdeterminants of $A$ are bounded in absolute value by a certain  $M\in \N$.
Suppose that $\vol(U_i)$ is less than half of the volume of the $d$-sphere. Then:
\[
\vol(U_{i+1})\ge \left(1 + \sqrt{\frac{2}{\pi}}\frac{1}{M^2 d^{2.5}}\right) \vol(U_i).
\]
\end{lemma}

From this the theorem is easily derived:

\begin{proof}[Proof of Theorem~\ref{thm:subdeterminants}]
The number of iterations needed to guarantee that $\vol(U_i)$ exceeds half of the ball is, by Lemma~\ref{lemma:volume-expansion},
bounded above by the smallest $i$ such that:
\[
\left(1 + \sqrt{\frac{2}{\pi}}\frac{1}{M^2 d^{2.5}}\right)^i \vol(c_v)\ge 2^d.
\]
Put differently, an upper bound for it is
\[
\frac{\ln(2^d/\vol(c_u))}{\ln\left(1 + \sqrt{\frac{2}{\pi}}\frac{1}{M^2 d^{2.5}}\right)}\ge
\sqrt{\frac{\pi}{2}} M^2 d^{2.5} \ln(2^d/\vol(c_u)).
\]
Now a lower bound is needed on the volume of an individual cone $c_u$. Such a bound is, for example, $1/(d!d^{d/2}M^d)$, since scaling down the generators of $c_v$ by a factor of $\sqrt{d}M$ makes them all be contained in the unit ball and since the volume enclosed by the generators is at least $1/n!$ (more precisely, it equals $1/d!$ times the determinant of the corresponding rows of $A$). This makes the number of needed iterations be in
\[
O(M^2 d^{2.5} \ln(2^d d!d^{d/2}M^d)) = O(M^2 d^{3.5} \ln(dM)),
\]
as claimed.
\qed
\end{proof}

The proof of Lemma~\ref{lemma:volume-expansion}, in turn, uses two ideas:
\begin{itemize}
\item For each individual cone it is possible to upper bound the ratio ``surface area to volume'' in terms of $M$ and $d$ (the precise bound the authors show for this ratio is $M^2d^3$).

\item For any union of vertex cones (in fact, for any spherical cone in dimension $d$) one has the following isoperimetric inequality: the ratio 
``surface area to volume'' of the whole cone is at least $\sqrt{2d/\pi}$.
\end{itemize}
These two inequalities are then combined as follows: when going from $U_i$ to $U_{i+1}$, the volume added by the new cones is at least 
$M^{-2}d^{-3}$ the total surface area of those cones. In turn, that covers at least the total surface area of $U_i$, which is at least $\sqrt{2d/\pi}$ times the volume of $U_i$.

That is:
\[
\vol(U_{i+1}\setminus U_i) \ge  \frac{\sqrt{2d/\pi}}{M^2d^3} \vol(U_i),
\]
which is the contents of Lemma~\ref{lemma:volume-expansion}.




\section{Conclusion}
This paper deals with the problem of bounding the diameter of polyhedra in terms of their dimension and number of facets. The last section of the paper revises recent progress (sometimes ``negative progress'', as in the case of the counter-examples to the Hirsch Conjecture) but most of the paper (Sections 2 and 3) is devoted to attempts at ``\emph{proving it by generalizing it}''. For this, the diameter problem on simplicial complexes is posed and studied without assuming that the complexes come from a polytope or polyhedron.

The main conclusion is that without extra conditions on the complexes polynomial bounds simply do not exist (Corollary 2.12). However, several results hint that there is hope of getting polynomial bounds under the mild assumption of the complexes being \emph{normal}, a.~k.~a.~\emph{locally strongly connected}. For example, Conjecture~\ref{conj:haehnle} would imply this, by Proposition~\ref{prop:clm}.

So, perhaps the main question is how plausible Conjecture~\ref{conj:haehnle} is. I have to admit that, although two years ago I was very optimistic about this conjecture, now I have doubts. Let me explain why.
For me the strongest evidence in favor of Conjecture~\ref{conj:haehnle} is the combination of Examples~\ref{exm:complete-clm} and~\ref{exm:injective-clm} and Proposition~\ref{prop:extremal-clm}. The examples show two extremal and ``opposite'' families of (multi)-complexes for which the conjecture holds and, what is more striking, for which the bound in the conjecture is sharp. What now makes me have doubts is the realization that these two examples are particular cases (in the ``connected layer family'' world) of flag and normal (multi)-complexes. And for flag and normal complexes we actually know the Hirsch bound to hold (Theorem~\ref{thm:hirsch-for-flag}).

On the other direction, the arguments of Section~\ref{sec:Hirsch-counter} indicate that we need new ideas if we want to have non-Hirsch polytopes that break the ``linear barrier''. Even the construction of polytopes (or pure, normal simplicial complexes, for that matter) with diameter exceeding $2n$ seems an extremely ambitious goal at this point.

\subsection*{Acknowledgements}
I would very much like to thank the two editors of TOP, Miguel \'Angel~Goberna and Jes\'us Artalejo, for the invitation to write this paper. Unfortunately, while the paper was being processed I received the extremely sad news of the passing away of Jes\'us. Let me send my warmest condolences to his colleagues, friends and, above all, his family. 

I also want to thank Jes\'us De Loera, Steve Klee, Tam\'as Terlaky, and Miguel \'Angel Goberna (again) for sending me comments and typos on the first version of the paper.

\end{document}